\documentclass{amsart}

\usepackage{amssymb}
\usepackage{graphicx}
\usepackage{bbm}
\usepackage{mathrsfs}
\usepackage{amsmath}
\usepackage{amsthm}
\usepackage[center]{caption}
\usepackage{enumerate}
\usepackage{verbatim}

\usepackage{marvosym}
\usepackage{color}
\usepackage[normalem]{ulem}
\usepackage{stmaryrd} 
\usepackage{soul}
\soulregister\cite7
\soulregister\citep7
\soulregister\citet7
\soulregister\ref7
\soulregister\pageref7

\reversemarginpar
\newcommand{\bbr}[1]{\ensuremath{\llbracket #1 \rrbracket}} 

\newcommand{\str}{\ensuremath{\mathrm{str}}}

\newcommand{\near}{\ensuremath{\alpha}}
\newcommand{\adjust}{\ensuremath{\beta}}
\newcommand{\isop}{\ensuremath{\gamma}}
\newcommand{\Pone}{Crawler}
\newcommand{\Ptwo}{Dasher}
\newcommand{\cC}{\ensuremath{\mathcal{C}}}
\newcommand{\cond}{\ensuremath{\eps}}
\newcommand{\inbd}{\ensuremath{\partial^i\hspace{-0.8mm}}}
\newcommand{\obd}{\ensuremath{\partial}}
\newcommand{\bneck}{\ensuremath{\theta}}
\newcommand{\lr}[1]{\stackrel{#1}{\longleftrightarrow}}
\newcommand{\lespar}{\lesssim_{\Delta,\eps,r}}
\newcommand{\eqpar}{\asymp_{\Delta,\eps,r}}
\newcommand{\gtrpar}{\gtrsim_{\Delta,\eps,r}}
\theoremstyle{plain}
\newtheorem{thm}{Theorem}[]
\newtheorem{cor}[thm]{Corollary}
\newtheorem{lem}[thm]{Lemma}
\newtheorem{prop}[thm]{Proposition}

\theoremstyle{definition}
\newtheorem{ex}{Example}
\newtheorem*{defn}{Definition}

\DeclareMathOperator*{\mix}{mix}
\DeclareMathOperator*{\stopop}{stop}
\DeclareMathOperator*{\hit}{hit}

\renewcommand{\P}{\mathbb{P}}
\newcommand{\E}{\mathbb{E}}

\newcommand{\hsl}{\hspace{1mm}}
\newcommand{\ind}{\mathbbm{1}}

\newcommand{\eps}{\varepsilon}
\newcommand{\bp}{\begin{proof}}
\newcommand{\ep}{\end{proof}}
\newcommand{\tmix}{t_{\mix}}
\newcommand{\tstop}{t_{\stopop}}
\newcommand{\thit}{t_{\hit}}
\newcommand{\Sc}{\mathcal{S}}

\newcommand{\Scal}{\mathcal{S}}
\newcommand{\calA}{\mathcal{A}}

\def\bal#1\eal{\begin{align*}#1\end{align*}}
\newcommand{\I}[1]{{\mathbbm{1}}_{\{#1\}}}

\numberwithin{equation}{section}
\numberwithin{thm}{section}

\author{Louigi Addario-Berry}
\address{Louigi Addario-Berry, McGill University, Mathematics and Statistics, Burnside Hall, 805 Sherbrooke Street West, Montreal H3A 0B9, Canada}
\email{louigi.addario@mcgill.ca}

\author{Matthew I.~Roberts}
\address{Matthew I.~Roberts, University of Bath, Department of Mathematical Sciences, Claverton Down, Bath BA1 1AR, UK}
\email{mattiroberts@gmail.com}

\title[Mixing time bounds via bottleneck sequences]{Mixing time bounds via bottleneck sequences}

\begin{document}

\begin{abstract}
We provide new upper bounds for mixing times of general finite Markov chains. We use these bounds to show that the total variation mixing time is robust under rough isometry for bounded degree graphs that are roughly isometric to trees.
\end{abstract}

\maketitle

\section{Introduction}

The mixing time is one of the most fundamental and well-studied quantities associated to a Markov chain. It is natural to ask how robust the mixing time is: how much can we change the mixing time by making small changes to the chain?

It has recently been shown, by Ding and Peres \cite{ding_peres:sensitivity_mixing} and by Hermon \cite{hermon:robustness_mixing} respectively, that neither total variation mixing times nor uniform mixing times are geometrically robust: in general, bounded perturbations of edge conductances can change both by arbitrarily large factors. For the moment, there is no general recipe which determines whether or not a given collection of graphs is robust in this sense.

One of the aims of this paper is to show that for a large class of chains---roughly speaking, any chain whose underlying graph is globally tree-like---the total variation mixing time is geometrically robust.

In order to do this we introduce new upper bounds on the mixing time that may be useful in their own right. It is well-known that the mixing time is related to bottlenecks in the graph, and our main idea is that the key quantity is the number and strength of bottlenecks that can be lined up in a row. We define a \emph{bottleneck sequence} to quantify this concept and use it heavily throughout the article.

The main results of this paper are Theorem \ref{RIequiv_new}, which is a statement about robustness of mixing times; Theorem \ref{upperbd1_new}, which gives an upper bound on mixing times using bottleneck sequences; and Theorem \ref{upperbd2}, which provides a stronger but more complicated upper bound on mixing times involving a game between two players.

\subsection{Notation}

Let $X=(X_n,n \ge 0)$ be an irreducible Markov chain on a finite state space $V$, and for $u,v\in V$ let $p_{uv}=\P_u(X_1=v)$, where $\P_u$ signifies that the chain starts from $u$. Let $E$ be the set of pairs $\{u,v\}$ with $u\neq v$ such either $p_{uv}>0$ or $p_{vu}>0$, and let $G=(V,E)$. Occasionally we will write $G_X$ to signify that the graph $G$ corresponds to the Markov chain $X$. To avoid periodicity issues, we assume throughout that $X$ is {\em lazy}, i.e.~that $p_{vv}=1/2$ for all $v\in V$. Write $\pi = (\pi(v), v\in V)$ for the stationary distribution of $X$.

We say that $X$ is {\em $\eps$-uniform} if $\eps\pi(x)p_{xy} \le \pi(u)p_{uv}\le \pi(x)p_{xy}/\eps$ for all $\{u,v\},\{x,y\}\in E$. In particular any irreducible $\eps$-uniform chain has $p_{uv}>0$ if and only if $p_{vu}>0$.

If $X$ is reversible, i.e. $\pi(x)p_{xy}=\pi(y)p_{yx}$ for all $xy\in E$, then we write $c_{xy}=\pi(x)p_{xy}$ for the \emph{conductance} of the edge $xy$.

A fundamental property of $X$ is the {\em total variation mixing time}, 
\[
t_{\mathrm{mix}}(X) = \min\{n \ge 0: \sup_v \|\P_v(X_n \in \cdot) - \pi\|  \le 1/4\},
\]
where $\|\mu-\nu\| = \max_{A \subset V} |\mu(A) - \nu(A)|$ denotes the total variation distance between probability measures $\mu$ and $\nu$ on $V$.

A \emph{stopping rule} for our Markov chain is a stopping time that is allowed to use extra randomness, as well as the history of the chain, to decide when to stop. For $v\in V$, we say that $\tau$ is a stopping rule from $v$ to $\pi$ if $\tau$ is a stopping rule and $\P_v(X_\tau = u) = \pi(u)$ for each $u\in V$. We define two more quantities that measure how long the chain takes to mix,
\[\tstop(X) = \max_{v\in V} \inf\{\E_v[\tau] : \tau \hbox{ is a stopping rule from $v$ to $\pi$}\}\]
and
\[\thit(\delta,X) = \max_{\substack{v\in V, A\subset V :\\ \pi(A)\geq \delta}} \E_v[H(A)]\]
where $H(A)$ is the first hitting time of $A$, $H(A)=\min\{n\ge 0 : X_n\in A\}$.

Aldous \cite[Theorem 6]{aldous:some} showed that there exists a constant $c_1$ such that for any lazy Markov chain $X$, and any $\delta>0$,
\begin{equation}\label{equiv1}
\delta\thit(\delta,X)\le c_1\tstop(X).
\end{equation}
He also showed \cite[Theorem 5]{aldous:some} that for lazy \emph{reversible} Markov chains,
\begin{equation}\label{equiv2}
\tmix(X) \asymp \tstop(X).
\end{equation}
Here the notation $\asymp$ means that there exist constants $0<c\le C<\infty$ such that $c\tmix(X)\leq \tstop(X)\leq C\tmix(X)$. We will use these bounds in our proofs. (In fact Aldous considered continuous time chains, but it is possible to adapt his proofs to discrete time, using the lazy nature of the chain for his Lemma 38.)

Let $d_G$ be the graph distance on $V$. To clarify, recall that an edge ${u,v}$ is in the graph if either $p_{uv}>0$ or $p_{vu}>0$, and $d_G(u,v)$ is the length of the shortest path from $u$ to $v$ in the graph; so $d_G$ is symmetric, and for example if $p_{uv}=0$ but $p_{vu}>0$ then $d_G(u,v)=1$. Let $\Delta(G)$ be the maximum degree over all vertices in $G$. For $A\subset V$ and $r\geq0$ write $B_G(r,A)=\{v\in V : d_G(v,A)\leq r\}$ for the closed ball of radius $r$ about $A$. A {\em correspondence} between graphs $G=(V,E)$ and $G'=(V',E')$ is a relation $\cC \subset V \times V'$ 
such that for all $v \in V$, there is $v' \in V'$ such that $(v,v') \in \cC$; and for all $v'\in V'$, there is $v\in V$ such that $(v,v')\in \cC$. In other words, the bipartite graph with vertices $V \cup V'$ and edges $\cC$ has no isolated vertices. The {\em stretch} of $\cC$ is 
\[
\mathrm{str}(\cC) = \sup\left\{\frac{d_{G'}(v',w')\vee d_G(v,w)+1}{d_{G'}(v',w')\wedge d_G(v,w)+1}: (v,v') \in \cC,(w,w') \in \cC\right\}\, . 
\]
We say $G$ and $G'$ are $r$-\emph{roughly isometric}, and write $G \simeq_r G'$,  if there is a correspondence $\cC$ between $G$ and $G'$ with $\mathrm{str}(\cC) \le r$. Our definition of rough isometry differs slightly from that given in e.g.~\cite{benjamini:coarse_geom}, but the reader may easily check that the two are equivalent up to adjusting the value of $r$. We prefer the current definition as it is obviously symmetric and will be easier to apply in our setting.

\subsection{Robustness of mixing}

\begin{thm}\label{RIequiv_new}
Fix a finite irreducible lazy Markov chain $X$ on a graph $G$, and another such chain, $Y$, on a tree $T$. Suppose that $X$ and $Y$ are both $\cond$-uniform and let $\Delta=\Delta(G)\vee\Delta(T)$. If $G \simeq_r T$ then
\[
\tstop(X)\asymp_{\Delta,\cond,r} \tstop(Y).
\]
If $X$ and $Y$ are reversible then
\[
\tmix(X)\asymp_{\Delta,\cond,r} \tmix(Y).
\]
\end{thm}
The notation $\asymp_{\Delta,\cond,r}$ indicates that e.g.~$(t_{\mathrm{mix}}(X) \vee t_{\mathrm{mix}}(Y))/(t_{\mathrm{mix}}(X)\wedge t_{\mathrm{mix}}(Y))$ is bounded by a function of $\Delta$, $\cond$ and $r$.

Example \ref{stars} in Section \ref{examples} shows that some dependence on $\Delta$ and $\cond$ is indeed necessary; obviously dependence on $r$ is also necessary. Furthermore, the aforementioned example of Ding and Peres \cite{ding_peres:sensitivity_mixing}, which we reproduce in Section \ref{examples} as Example \ref{DPex}, shows that for fixed $\Delta$ and $\cond$ the mixing time is not robust under $r$-rough isometry for any $r$, so the condition that $T$ must be a tree cannot be removed entirely.

\subsection{Bottleneck sequences}

Our strategy for proving Theorem \ref{RIequiv_new} is to give a geometric characterization of the mixing time for treelike graphs. In order to do this we need some further definitions. For $A,B \subset V$ define  
\[
Q(A,B) = \P_\pi(X_0 \in A, X_1 \in B)\quad\mbox{ and }\quad 
\Phi(A) = \frac{Q(A,A^c)}{\pi(A)\pi(A^c)}.
\] 
These quantities capture how easy or difficult it is for $X$ to move between different subsets of the state space. If $A=\{a\}$ is a singleton we write $Q(a,B)$ instead of $Q(\{a\},B)$, and likewise write $Q(A,b)$ and $Q(a,b)$. In the remainder of the paper, we use this convention without comment when applying other set functions to singletons.

Note that for any Markov chain and any $A\subset V$, $Q(V,A) = \P_\pi(X_1\in A) = \P_\pi(X_0\in A)=Q(A,V)$ and therefore
\begin{equation}\label{eq:Qswap}
Q(A,A^c) = Q(A,V)-Q(A,A) = Q(V,A)-Q(A,A) = Q(A^c,A).
\end{equation}

For a set $A \subset V$ of vertices, we write ``$A$ is connected'' to mean that the induced subgraph $G[A]$ is connected. To clarify, connected means that for any $u,v\in A$ there is a path $u_0,\ldots,u_l$ from $u$ to $v$ within $A$ such that for each $i$, either $p_{u_iu_{i+1}}>0$ or $p_{u_{i+1}u_i}>0$.

We define $\obd A = \{u\in A^c:\hsl \exists v\in A \hbox{ with } p_{uv}>0\}$, and $\inbd A = \{v\in A : \hsl \exists u\in A^c \hbox{ with } p_{uv}>0\}$.

Given a Markov chain $X$ and $\bneck \in (0,1]$, a {\em $\bneck$-bottleneck sequence} for $X$ is an increasing sequence $S_1\subset S_2\subset\ldots\subset S_l$ of subsets of $V$ with $S_1 \ne \emptyset$ and $S_l\ne V$ such that 
\begin{itemize}
\item $S_j$ and $S_j^c$ are both connected for each $j=1,\ldots, l$;
\item $Q(S_{j+1}\setminus S_j,S_j) \ge \bneck Q(S_j^c,S_j)$.
\end{itemize}
The second condition says that, in stationarity, when a random walk enters $S_j$ it is reasonably likely to have come from $S_{j+1}\setminus S_j$. If $\bneck=1$ then it states that $\partial S_j \subset S_{j+1}$. For any $\bneck > 0$ it implies that $\partial{S_j} \cap S_{j+1}$ is non-empty. 

Let $\mathcal{S}_{\bneck}=\mathcal{S}_{\bneck}(X)$ be the set of $\bneck$-bottleneck sequences for $X$.

Our proofs are inspired by the approach of Lovasz and Kannan \cite{lovasz_kannan:faster_mixing}, who proved bounds on $\tmix(X)$ for reversible chains by considering bottlenecks at multiple scales, and of Fountoulakis and Reed \cite{fountoulakis_reed:faster_mixing}, who showed that connectivity could be exploited to strengthen the Lovasz--Kannan bounds. The next result further improves the bound in \cite{fountoulakis_reed:faster_mixing} by considering only nested sequences of bottlenecks, rather than all possible bottlenecks at each scale. We emphasise that it does not require the underlying graph to be tree-like.

\begin{thm}\label{upperbd1_new}
For any finite irreducible lazy Markov chain $X$, and any $\bneck \in (0,1)$,
\[
\tstop(X) \lesssim_{\bneck} \max_{(S_1,\ldots,S_l) \in \Scal_\bneck(X)} \sum_{j=1}^{l} \frac{1}{\Phi(S_j)}. 
\]
If $X$ is also reversible, then
\[
\tmix(X) \lesssim_{\bneck} \max_{(S_1,\ldots,S_l) \in \Scal_\bneck(X)} \sum_{j=1}^{l} \frac{1}{\Phi(S_j)}. 
\]
\end{thm}
The notation $\lesssim_{\bneck}$ means that e.g.~$\tstop(X)/\max_{(S_1,\ldots,S_l) \in \Scal_\bneck(X)} \sum_{j=1}^{l} \frac{1}{\Phi(S_j)}$ is bounded from above by a function of $\bneck$ only, uniformly in $X$.

In Section~\ref{sec:related} we describe the relation between Theorem~\ref{upperbd1_new} and existing results. In particular, we explain the Fountoulakis-Reed result, and how it follows from Theorem~\ref{upperbd1_new}.

In order to prove Theorem \ref{RIequiv_new}, we strengthen Theorem~\ref{upperbd1_new} by showing that we need not maximize over all bottleneck sequences. We again bound $\tmix(X)$ by an expression of the form $\sum 1/\Phi(D_j)$, for some sequence of bottlenecks $(D_j)$. However, the sequence is chosen according to a \emph{game}, rather than simply taking the worst possible sequence as in Theorem~\ref{upperbd1_new}. Informally, this allows us to choose some of the points near $D_j$ and force $D_{j+1}$ to contain those points. This means that it is possible to---roughly speaking---force our bottleneck sequence to move in a particular direction.

\subsection{The bottleneck sequence game}\label{gameintro}
Fix a finite irreducible lazy Markov chain $X$ as above. We describe a game played between two players, which builds an increasing sequence $D_1,\ldots,D_l$ of subsets of $V$. One player, called \Pone, aims to maximise $\sum_{k=1}^l 1/\Phi(D_k)$, usually by making the sequence advance slowly; the other player, \Ptwo, aims to minimize the same quantity by growing the sequence quickly. 

We prove that the rules of the game imply that whatever strategy \Ptwo~adopts, \Pone~can make $\sum_{k=1}^l 1/\Phi(D_k)$ larger than $\tstop(X)$, up to a constant factor. This is formalized in Theorem \ref{upperbd2}. 
Knowing this, we can then bound $\tmix(G)$ from above by choosing a specific strategy for \Ptwo, then proving upper bounds on $\sum_{k=1}^l 1/\Phi(D_k)$ when \Ptwo~follows this strategy.  This is how we will prove Theorem~\ref{RIequiv_new}. 

Defining the game requires a few more definitions. For a set $A \subset V$, recall that $H(A)=\min\{n: X_n \in A\}$ is the first hitting time of $A$ by $X$, and write $H^+(A) = \min\{n \ge 1: X_n \in A\}$.

\begin{defn}
For $u,v \in V$ and $\near \in (0,1]$, we say that $v$ is $\near$-near to $u$ if
\[\pi(v)\P_v(H(u) \le1/\near)\ge\near \pi(u).\]
For $A,B\subset V$, we say that $A$ is $\near$-near to $B$ if every vertex of $A$ is $\near$-near to some vertex of $B$.
\end{defn}
Note that if $v\in A$ then clearly $v$ is $\near$-near to $A$ for every $\near$.

\begin{defn}
Fix sets $A \subset B \subset V$ with $A$ connected, and $\adjust \in (0,1]$. We say that $B$ is a \emph{$\adjust$-adjustment} of $A$ if for any set $S \subset B^c$ such that $A\cup S$ is connected, it holds that 
\[
Q((B \cup S)^c,A\cup S) \ge \adjust Q((B \cup S)^c,B \cup S)\, .
\]
\end{defn}
Heuristically, if $B$ is not much bigger than $A$, and removing $B\setminus A$ from the graph does not ``create low-conductance sets containing $A$'', then $B$ is an adjustment of $A$.

If $B$ is a $\adjust$-adjustment of $A$, taking $S =\emptyset$ shows that $Q(B^c,A) \ge \adjust Q(B^c,B)$. This states that in stationarity, when the random walk enters $B$, it is reasonably likely to enter $A$ at the same moment. However, this condition is not sufficient to imply that $B$ is a $\adjust$-adjustment. To see an example, consider simple random walk on the path $\{1,\ldots,n\}$ for $n\geq 4$. Let $A = \{2\}$ and $B=\{2,3\}$. Then $Q(B^c,A) = Q(B^c,B)/2$. However, if we take $S=\{1\}$, then $B$ separates $A\cup S$ from the rest of the graph so $Q((B\cup S)^c,A\cup S)=0<Q((B\cup S)^c,B\cup S)$ and we see that $B$ is not a $\adjust$-adjustment of $A$ for any $\adjust\in(0,1]$.

Fix a vertex $s \in V$, and $\near,\adjust\in(0,1]$ such that $1/\near\in\mathbb{N}$, and $\isop\in (0,1)$. We now describe the rules of the $(s,\near,\adjust,\isop)$-bottleneck sequence game for $X$.
Recall that the players are called \Pone~and \Ptwo. 
A position of the game is a pair $(C,D)$ of subsets of $V$, or equivalently an element of $2^{V} \times 2^{V}$. We start from $(\emptyset,\emptyset)$. Play alternates, starting with \Pone.

From position $(C,D)$, a $\isop$-valid move for \Pone~is any position $(C',D)$ satisfying 
\begin{enumerate}[(a)]
\item {\bf (Connectivity)} $C\subset C'$, $C'\setminus C \subset D^c$, and $C'$ is connected.
\item {\bf (Isoperimetry)} $Q((D\cup C')^c,C) \leq \isop Q(D^c,C)$.
\end{enumerate}

From position $(C,D)$, an $(s,\near,\adjust)$-valid move for \Ptwo~is any position $(C,D')$ satisfying 
\begin{enumerate}[(i)]
\item {\bf (Complement connectivity)} $D\cup C\subset D'$ and $(D')^c$ is connected.
\item {\bf (Nearness)} $\inbd D'$ is $\near$-near to $C$.
\item {\bf (Adjustment)} $D'$ is a $\adjust$-adjustment of $C$.
\item {\bf (Endgame)} If $s\in D'$, then $s$ is $\near$-near to $C$ and $D'=V(G)$.
\end{enumerate}
The game ends as soon as \Ptwo~chooses $D'=V(G)$.

Say a sequence $((C_i,D_i),0 \le i \le l+1)$ is $(s,\near,\adjust,\isop)$-valid  
if $(C_0,D_0)=(\emptyset,\emptyset)$ and  
\begin{itemize}
\item for each $i \le l$, $(C_{i+1},D_i)$ is a $\isop$-valid move for \Pone~starting from $(C_i,D_i)$,
\item for each $i \le l$, $(C_{i+1},D_{i+1})$ is a $(\near,\adjust,s)$-valid move for \Ptwo~starting from $(C_{i+1},D_i)$. 
\end{itemize}
We say the sequence is an $(s,\near,\adjust,\isop)$-valid {\em game} if it is $(s,\near,\adjust,\isop)$-valid  and additionally $D_{l+1}=V$. 

Note that if $((C_i,D_i),0 \le i \le l+1)$ is a valid game then by isoperimetry, $C_{i+1}\setminus D_i$ is non-empty for each $i \le l$. In particular, $D_{l} \ne V$.

In Lemma~\ref{welldef_new} we shall see that the rules are consistent, in that both players can always make a valid move (given that the previous moves have been valid).  We now state the upper bound on the mixing time that arises from the bottleneck sequence game, which we will prove in Section \ref{gameproofsec}. Again we emphasise that this bound holds for any graph, not just tree-like graphs.

\begin{thm}\label{upperbd2}
Fix a finite irreducible lazy Markov chain $X$, and any $\near,\adjust\in(0,1]$ and $\isop\in(0,1)$. There exist $s\in V$ and $\gamma$-valid moves for \Pone~such that, whatever $(s,\near,\adjust)$-valid moves \Ptwo~makes,
\[\tstop(X) \lesssim \frac{1}{\near^3} + \frac{1}{\near^2\adjust\isop} \sum_{k=1}^l \frac{1}{\Phi(D_k)}.\]
If $X$ is also reversible, then
\[\tmix(X) \lesssim \frac{1}{\near^3} + \frac{1}{\near^2\adjust\isop} \sum_{k=1}^l \frac{1}{\Phi(D_k)}.\]
\end{thm}

We now briefly explain how to see that Theorem \ref{upperbd2} implies Theorem \ref{upperbd1_new}. Given $A,B,C\subset V$, write $A \lr{C} B$ if every path from $A$ to $B$ contains a vertex of $C$.
Note that if $a \in C$ then $a \lr{C} B$ for all $B \subset V$. 

\begin{defn}
The $s$-\emph{hull} of $A$ is the set 
$h_s(A) = \{t: s \lr{A} t\}\supset A$ of vertices disconnected from $s$ when $A$ is removed.
\end{defn}

We shall see in Lemma \ref{welldef_new} that in the course of the game, if the position is $(C,D)$ and it is \Ptwo's move, then \Ptwo~can always move to $(C,h_s(C\cup D))$.
If \Ptwo~follows this strategy, then $D_1,D_2,\ldots,D_l$ is a $(1-\isop)$-bottleneck sequence.
Therefore Theorem \ref{upperbd2} implies that there exist valid moves $C_1,\ldots, C_l$ for \Pone~such that, if \Ptwo~always moves to $D_k = h_s(C_k\cup D_{k-1})$, then
\[\tmix(X)\lesssim 1+\frac{1}{\isop}\sum_{k=1}^l \frac{1}{\Phi(D_k)} \lesssim \frac{1}{\isop} \max_{(S_1,\ldots,S_l) \in \Scal_\isop(X)} \sum_{j=1}^{l} \frac{1}{\Phi(S_j)}.\]
Theorem \ref{upperbd2} is therefore, up to constants, at least as strong as Theorem \ref{upperbd1_new}. In some cases it is stronger: Example \ref{ub1nottightnew} in Section \ref{examples} gives a graph where Theorem \ref{upperbd1_new} yields an upper bound on the mixing time of order $n^4$, whereas Theorem \ref{upperbd2} gives $n^3$, which is the correct order.

In Section \ref{thm1proof} we use Theorem \ref{upperbd2} to prove Theorem \ref{RIequiv_new}.

\subsection{Related work}\label{sec:related}

Ding, Lee and Peres \cite{ding_lee_peres:cover_times} showed that the expected {\em cover time} of reversible chains is determined, up to constant factors, by a functional of the graph geometry. Write $\tau_{\mathrm{cov}}$ for the first time every vertex of $V$ has been visited by $X$. The cover time for $X$ is defined as $t_{\mathrm{cov}}(X) = \max_{v \in V} \E_v(\tau_{\mathrm{cov}})$.

Fix any $v_0 \in V$ and let $\{\eta_v\}_{v \in V}$ be the Gaussian free field (GFF) on $G$ with $\eta_{v_0}=0$. This is the centered Gaussian process whose covariance matrix is determined by the identities $\E[(\eta_u-\eta_v)^2] = R_{\mathrm{eff}}(u,v)$ for $u,v \in V$. Here, $R_{\mathrm{eff}}(u,v)$ is the effective resistance between $u$ and $v$ when $G$ is viewed as an electrical network with conductances $(c_{uv},\, uv\in E)$ (recall that $c_{uv}=\pi(u)p_{uv}$). Ding, Lee and Peres showed that 
\[
t_{\mathrm{cov}}(X) \asymp \E\big[ \max_{v \in V} \eta_v\big]^2\, .
\]
The maximum of the GFF is a geometric functional, so we may view the cover time as ``geometrically robust''.

Returning to mixing times, Oliveira \cite{oliveira:mixing} and independently Peres and Sousi \cite{peres:mixing} showed that for reversible chains, not only does \eqref{equiv1} hold, but in fact $\tmix(X) \asymp_\delta \thit(\delta,X)$ for any $\delta<1/2$. Peres and Sousi then used this equivalence to show that if $X$ and $Y$ are two lazy reversible Markov chains on the same tree $T$ with conductances bounded above and below by some strictly positive constants, then
\begin{equation}\label{eq:tree_equiv}
\tmix(X)\asymp\tmix(Y).
\end{equation}
Our Theorem \ref{upperbd1_new} yields another simple proof of (\ref{eq:tree_equiv}): see Corollary \ref{treecor}. The equivalence $\tmix(X) \asymp_\delta \thit(\delta,X)$ in the more delicate case $\delta=1/2$ was later established by Griffiths et al.~\cite{griffiths_et_al:tight_hitting_times}.

Fountoulakis and Reed \cite{fountoulakis_reed:faster_mixing}, building on work by Jerrum and Sinclair \cite{jerrum1988conductance} and Lov\'asz and Kannan \cite{lovasz_kannan:faster_mixing}, showed that if $X$ is reversible and we set
\[\Phi(p) = \min_{\substack{A \text{ connected},\\ p/2\leq \pi(A)\leq p}} \Phi(A) \hspace{6mm} \hbox{ and } \hspace{6mm} m=\max_{v\in V} \lfloor\log_2 \pi(v)^{-1}\rfloor,\]
then
\begin{equation}\label{FRbd}
\tmix \lesssim \sum_{j=1}^m \frac{1}{\Phi(2^{-j})^2}.
\end{equation}
Our proofs are largely based on the approach of Fountoulakis and Reed \cite{fountoulakis_reed:faster_mixing}. In fact, by following the end of the proof in \cite{fountoulakis_reed:faster_mixing}, we may deduce their bound (\ref{FRbd}) from our Theorem \ref{upperbd1_new}. Example \ref{stars} in Section \ref{examples} gives a Markov chain for which Theorem \ref{upperbd1_new} gives the correct order of magnitude of $\tmix$, whereas (\ref{FRbd}) does not.

Another variant on the Lov\'asz and Kannan result was obtained by Morris and Peres \cite{morris_peres:evolving_sets}, who sharpened \eqref{FRbd} in several ways, including allowing non-reversible chains and bounding the larger $L^\infty$-mixing time. This mantle was then taken up by Goel, Montenegro and Tetali \cite{goel_et_al:mixing_spectral}, who used the spectral profile $\Lambda(p)$ instead of the conductance profile $\Phi(p)$. For reversible chains $\Lambda(p)$ may be thought of as the smallest eigenvalue of the Laplacian on the graph restricted to $A$, minimized over all sets $A\subset V$ of invariant measure at most $p$; in fact the definition of $\Lambda(p)$ used in \cite{goel_et_al:mixing_spectral} is different, but for reversible chains it is equivalent to this definition up to constants provided that $p\le 1/2$. Goel, Montenegro and Tetali gave an upper bound involving $\Lambda(p)$, and were able to recover the result of Morris and Peres \cite{morris_peres:evolving_sets} in the reversible case by using a discrete Cheeger inequality to relate $\Lambda(p)$ to $\Phi(p)$.

Kozma \cite{kozma:precision_spectral_profile} showed that the upper bound on the $L^\infty$-mixing time given in \cite{goel_et_al:mixing_spectral} is not always correct up to constant factors. He then asked the general question ``is the mixing time a geometric property?'' and conjectured that the mixing time was robust (up to constant factors) under rough isometry for bounded degree graphs.

However, a construction of Ding and Peres \cite{ding_peres:sensitivity_mixing} shows that \eqref{eq:tree_equiv} cannot hold in general if the underlying graph is not a tree, even if it has bounded degree. The message that we take from this is that the total variation mixing time is \emph{not} geometrically robust in general. One of the main aims of this article is to show that there is robustness amongst a wider class than just trees: indeed, \eqref{eq:tree_equiv} holds if the graph has bounded degree and is roughly isometric to a tree. We will explain Ding and Peres' construction in more detail in Section \ref{examples}.

Hermon \cite{hermon:robustness_mixing} has recently shown that the $L^\infty$-mixing time is also not geometrically robust over all bounded degree graphs. His construction combines aspects of Kozma's and Ding and Peres' examples.

\subsection{Layout of the article}

In Section \ref{examples} we give several examples of Markov chains that highlight the key features of our results, as well as their limitations. We then begin our proofs, starting in Section \ref{trees} with an easy lower bound on the mixing time for Markov chains on trees, which is essentially already known but which gives a framework for proving a similar lower bound on graphs that are roughly isometric to trees. In Section \ref{t2proof} we prove Theorem \ref{upperbd1_new}, which both serves as a warm-up for the proof of Theorem \ref{upperbd2} and combines with the work in Section \ref{trees} to show robustness for the mixing time on trees (see Corollary \ref{treecor}). In Section \ref{gen_lb_sec} we generalise the approach in Section \ref{trees} to give a lower bound for the mixing time on graphs that are roughly isometric to trees, and then in Section \ref{gameproofsec} we prove Theorem \ref{upperbd2}. Finally we combine our work from Sections \ref{gen_lb_sec} and \ref{gameproofsec} to prove Theorem \ref{RIequiv_new} in Section \ref{thm1proof}.

\section{Examples}\label{examples}

In this section we construct some illuminating examples, several of which are referred to elsewhere in the article. We begin with two trees that are roughly isometric, but which have very different mixing times. This example shows that Theorem \ref{RIequiv_new} cannot hold without control over the maximum degree of both $G$ and $T$. It also provides an example where Theorem \ref{upperbd1_new} gives the correct order of magnitude for $\tmix$, but (\ref{FRbd}) does not.

\begin{ex}\label{stars}
Take two star graphs of order $n$ (that is, trees with one internal node and $n$ leaves). Choose one leaf from each, $v_1$ and $v_2$, and join them by a single edge. Put unit conductance on each edge. This tree $T$ has mixing time of order $n$: started from any vertex other than $v_1$ or $v_2$, it takes time of order $n$ to reach the other half of the graph. However, $T$ is 3-roughly isometric to the tree $T'$ consisting of two nodes joined by a single edge with unit conductance, which obviously has mixing time of order $1$.

It is also easy to check that the bound (\ref{FRbd}) gives $\tmix(T) \lesssim n^2$, whereas Theorem \ref{upperbd1_new} gives $\tmix(T)\lesssim n$.

Similarly, by considering a path of length $3$ whose middle edge has conductance $1$ and whose outer edges both have conductance $n$, which is again roughly isometric to $T'$, we see that Theorem \ref{RIequiv_new} cannot hold without some control over the conductance ratio $\cond$.
\end{ex}

The next example shows that the bound in Theorem \ref{upperbd1_new} need not hold when $\theta=1$.

\begin{ex}\label{gamma1}
Take a complete graph on $n$ vertices, and choose one distinguished Hamiltonian cycle. If $uv$ is in the distinguished cycle, then set $c_{uv}= 1$, and otherwise set $c_{uv} = 1/n^3$. Then the lazy random walk with these conductances takes time of order $n^2$ to mix, but $\max_{(S_1,\ldots,S_l) \in \Scal_1(X)} \sum_{j=1}^{l} \frac{1}{\Phi(S_j)}$ is of order $n$.
\end{ex}

A graph is roughly isometric to a tree if and only if it has bounded connected tree-width. We refer to \cite{diestel:connected_tree_width} for an explanation of tree-width. In considering whether Theorem \ref{RIequiv_new} is best possible, it is natural to ask whether we might be able to say that every graph of bounded tree-width has mixing time within a constant factor of its tree decomposition, without requiring the parts of the tree decomposition to be connected. The following example shows that this is not possible.

\begin{ex}
The cycle of length $n$ has a tree decomposition of bounded width whose parts form a binary tree of depth $\log_2 n$. However the mixing time of lazy simple random walk on the cycle is of order $n^2$, whereas on the binary tree of depth $\log_2 n$ it is of order $n$.
\end{ex}

The next example is due to Ding and Peres \cite{ding_peres:sensitivity_mixing}. It shows that Theorem \ref{RIequiv_new} cannot hold for general bounded degree graphs without some structural assumption, as well as showing that $\max_{(S_1,\ldots,S_l) \in \Scal_1(X)} \sum_{j=1}^{l} \frac{1}{\Phi(S_j)}$ is not always a lower bound for $\tmix$. We give a non-rigorous discussion and refer to \cite{ding_peres:sensitivity_mixing} for the details.

\begin{ex}[Ding, Peres \cite{ding_peres:sensitivity_mixing}]\label{DPex}
Take a binary tree $T$ of height $K$ rooted at $o$. Distinguish the two children of each vertex as the {\em left} and {\em right} children. Let $L$ and $R$ be the sets of left and right children respectively. Let $\Gamma_{u,v}$ be the unique path in $T$ from $u$ to $v$. Define
\[B = \big\{ v\in T : K/4 \leq |\Gamma_{o,v}| \leq K/2, \hbox{ and } \big||\Gamma_{o,v}\cap L|-|\Gamma_{o,v}\cap R|\big| \leq \sqrt K\big\}\]
to be the set of {\em balanced} vertices. To every vertex in $B$, attach a path of length $K$. Finally, attach an expander of size $K^2 2^K$ to the leaves of $T$, in such a way that every leaf is joined to a different vertex in the expander.

It is easy to see that the mixing time is of the same order as the maximum hitting time of the expander over all starting vertices. If every edge has unit conductance, then starting from the root, with high probability we hit of order $K$ balanced vertices, and spend time of order $K$ in each of the attached paths. Therefore the mixing time is at least $K^2$. On the other hand, if we change the conductance on every edge $(u,v)$ with $v\in L$ to $2$, then with high probability we hit at most of order $\sqrt K$ balanced vertices before hitting the expander, regardless of where we start. Thus the mixing time is at most $K^{3/2}$.

Note also that taking $S_j$ to be the first $j$ levels of $T$ along with any attached paths  for $j=1,\ldots,K-1$, and setting $l=K-1$, we have
\[\sum_{j=1}^{l} \frac{1}{\Phi(S_j)} \gtrsim \sum_{j=K/4}^{K/2} \frac{1}{\Phi(S_j)} \asymp \sum_{j=K/4}^{K/2} K \asymp K^2,\]
provided the conductances all fall between two positive constants. As we have just seen, this is not a lower bound for $\tmix$.
\end{ex}

Finally we construct an example where Theorem \ref{upperbd2} gives the correct order for $\tmix$ and Theorem \ref{upperbd1_new} does not. 

The Cartesian product $G\times H$ of graphs $G$ and $H$ is the graph with vertices $V(G) \times V(H)$ and edges 
$\{(u,i)(v,j): \hbox{either } u=v \hbox{ and } ij \in E(H) \hbox{, or } i=j \hbox{ and } uv \in E(G)\}$. 

\begin{ex}\label{ub1nottightnew}
Let $Q=K_k \times C_n$, where $K_k$ is a complete graph with $k$ vertices and $C_n$ is a cycle with $n$ vertices; for concreteness say $C_n$ has edges $c_1c_2,c_2c_3,\ldots,c_nc_1$. We say a node of $Q$ has level $l$ if it is an element of the $k$-clique corresponding to $c_l$. For this example we think of $k$ as small and $n$ large.

Let $Q_1,\ldots,Q_n$ be disjoint copies of $Q$, and for each $i \le n$ fix a node $v_i$ at level $1$ of $Q_i$, and another, $v_i'$, at level $n$ of $Q_i$. 
Then create a graph $G$ by adding an edge between $v_i$ and $v_{i+1}$, and another between $v_i'$ and $v_{i+1}'$, for each $i \le n-1$. The graph $G$ has $kn^2$ vertices and maximum degree $k+3$.

Now define a sequence $(S_l,\,l < n^2)$ as follows. 
For $i,j \le n$ let $S_{j+n(i-1)}$ contain all nodes of $Q_1,\ldots,Q_{i-1}$, and all nodes of $Q_i$ whose level is at most $j$. When $l$ is a multiple of $n$ the only edges exiting $S_{l}$ lead to $S_{l+1}$. When $l$ is {\em not} a multiple of $n$, the set $S_{l}$ has $2k+2$ edges to $S_{l}^c$, and all but two of these lead to $S_{l+1}$. It follows that with unit conductance on every edge, $(S_l,\,l=1,\ldots, n^2-1)$ is a $\theta$-bottleneck sequence for any $\theta \le (2k-1)/(2k+1)$.
Moreover, when $l \le n^2/2$ we have $\Phi(S_l) \le 2/l$, so 
\[
\sum_{l} \frac{1}{\Phi(S_l)} \gtrsim n^4\, .
\]
However, in reality, the random walk spends on average time $k^2 n$ wandering around each copy of $Q$ each time it visits it, and apart from these detours it traverses a path of length $n$; so its mixing time is of the order $k^2 n\times n^2 = k^2 n^3$, much smaller than $n^4$ when $k^2\ll n$.

Indeed, in the bottleneck sequence game, as soon as \Pone~uses vertex $v_i$ in $C$, \Ptwo~can add the whole of the rest of $Q_i$ to $D$. In this way the game consists of at most $n$ moves with $\Phi(D_l)\geq 1/(2l k^2n)$ for each $l$; so Theorem \ref{upperbd2} gives an upper bound on $\tmix$ of the order $k^2 n^3$, which agrees with the actual mixing time.
\end{ex}

\section{An easy lower bound for the mixing time on trees}\label{trees}

The starting point for bounding the mixing time from below in terms of conductances is the inequality
\[\tmix \gtrsim \max_{S : \pi(S)\le 1/2}\frac{1}{\Phi(S)}.\]
This is based on the observation that in order to mix, we have to cross the worst bottleneck in our graph. For a proof, see for example \cite[Theorem 7.3]{levin_et_al:MCs_mixing}.

If our underlying graph is a tree, then we may use the recursive structure to get a sequence of bottlenecks, all of which must be crossed in order to mix. This idea allows us to give a simple lower bound for the mixing time, which is essentially a consequence of Moon's theorem \cite[Theorem 4.1]{moon:random_walks_trees}. We will develop this approach in Section \ref{gen_lb_sec} to cover graphs that are roughly isometric to trees.

\begin{prop}\label{treelower}
Suppose that $X$ is a Markov chain on a tree $T=(V,E)$. Then
\[\tmix(X) \gtrsim \max_{(S_1,\ldots,S_l) \in \Sc_1(X)} \sum_{i=1}^{l} \frac{1}{\Phi(S_i)}.\]
\end{prop}

\begin{proof}
We note first that every Markov chain on a tree is reversible (indeed, removing any edge $uv$ splits the tree into two connected components $A$ and $A^c$, so using \eqref{eq:Qswap}, $\pi(u)p_{uv} = Q(A,A^c) = Q(A^c,A) = \pi(v)p_{vu}$) so we may write $c_{uv} = \pi(u)p_{uv} = c_{vu}$ for the conductance of edge $uv$, and $c_u=\sum_{v\in V} c_{uv} = \pi(u)$.

Suppose $A\subset V$ is such that both $A$ and $A^c$ are connected. Since $T$ is a tree, there is exactly one edge $uv$ with $u\in A$ and $v\in A^c$.

The observation that $\E_v[H^+(v)\,|\,X_1=u]=\E_u[H(v)] +1$, and the identity (see \cite[Lemma 2.5]{aldous:reversible}) $\E_v[H^+(v)]=\pi(v)^{-1}$ applied in the subtree of $T$ with vertices $A\cup\{v\}$, together yield that
\[
\E_u[H(v)] + 1 = \frac{c_{uv} + \sum_{x\in A}c_x}{c_{uv}} = \frac{Q(A,A^c)+\pi(A)}{Q(A^c,A)} = 1 + \frac{\pi(A)}{Q(A,A^c)}.
\]
Thus $\E_u[H(v)] \ge 1/\Phi(A)$.

Choose a 1-bottleneck sequence $S_1,\ldots, S_l$ for $X$. Let $L=\max\{i : \pi(S_i)\le 1/2\}$. Since each set $S_i$ is connected and has connected complement, for each $i \le l$ there is a unique edge $s_it_i$ with $s_i \in S_i$ and $t_i \in S_i^c$. 
If the Markov chain starts in $S_1$, then in order to visit $S_L^c$ the chain must cross each of the edges $s_it_i$ for $i \le L$. It follows by \eqref{equiv1} and \eqref{equiv2} and the above that for any $v \in S_1$, 
\[
\tmix\gtrsim \E_v[H(S_L^c)] \ge \sum_{i =1}^L \E_{s_i}[H(t_i)] \ge 
\sum_{i =1}^L \frac{1}{\Phi(S_i)}\, .
\]
Finally, note that the sequence $(S_l^c,S_{l-1}^c,\ldots,S_1^c)$ is also a $1$-bottleneck sequence; applying the above argument to this sequence gives 
$\tmix\gtrsim\sum_{i=L+1}^l 1/\Phi(S_i)$, and the result follows.
\end{proof}

\section{Upper bound using bottleneck sequences: proof of Theorem \ref{upperbd1_new}}\label{t2proof}

Although Theorem \ref{upperbd1_new} is essentially a special case of Theorem \ref{upperbd2}, its proof contains several ideas (and two lemmas and a corollary) that we will need later, so we include it as a warm-up. We recall that for $s\in V$, a stopping rule $\tau$ from $s$ to $\pi$ is simply a stopping time that may use additional randomness to decide when to stop, and satisfies $\P_s(X_\tau = v) = \pi(v)$ for all $v\in V$. It is known (see \cite{lovasz_winkler:efficient_stopping_rules}) that \emph{optimal} stopping rules always exist, i.e.~there is a stopping rule $\tau$ satisfying $\E_s[\tau] = \inf\{\E_s[\tau'] : \tau' \hbox{ is a stopping rule from $s$ to $\pi$}\}$. By (\ref{equiv2}), Theorem \ref{upperbd1_new} is therefore a consequence of the following result.

\begin{prop}\label{upperbd}
Suppose that $X$ is a Markov chain on a graph $G=(V,E)$ with invariant measure $\pi$. Fix $\theta\in(0,1)$. For any $s\in V$ and any optimal stopping rule $\tau$ from $s$ to $\pi$,
\[\E_s[\tau] < \frac{1}{1-\theta} \max_{(S_1,\ldots,S_l) \in \Scal_\theta(X)} \sum_{j=1}^{l} \frac{1}{\Phi(S_j)}.\]
\end{prop}

Our aim now is to prove Proposition \ref{upperbd}. Our strategy is similar to that of Fountoulakis and Reed \cite{fountoulakis_reed:faster_mixing}.

Fix $s\in V$ and an optimal stopping rule $\tau$ from $s$ to $\pi$. For $v\in V$, define the {\em scaled exit frequencies} $\{y_v, v \in V\}$ by 
\[
y_v = \frac{1}{\pi(v)}\E_s[\#\{k < \tau: X_k = v\}]\, .
\]
(Since we will be keeping $s$ and $\tau$ fixed, we omit them from the notation.) Label the elements of $V$ as $1,\ldots,N$ so that $y_1 \le \ldots \leq y_N$. 
We will bound $\E_s[\tau]$ using the identity $\E_s[\tau]= \sum_{v=1}^N \pi(v)y_v$.

Since $\tau$ is an optimal stopping rule, it has a {\em halting state} $h$ such that $y_h=0$ (again see \cite{lovasz_winkler:efficient_stopping_rules}). Thus $y_h=y_1=0$ and we may certainly choose the ordering above so that $h=1$. The following identity is key to our analysis. It is originally due to Lovasz and Kannan \cite{lovasz_kannan:faster_mixing}, but we include a short proof as our formulation is slightly different; in particular we allow non-reversible chains.

\begin{lem}\label{lem:pizsum}
If $Z \subset V$ and $s \not\in Z$ then 
\[\pi(Z) = \sum_{u \not\in Z, v \in Z} (y_u Q(u,v) - y_v Q(v,u)).\]
\end{lem}

\begin{proof}
For all $v \in V$, $X_\tau = v$ if and only if 
$\#\{k \le \tau: X_k=v\} = \#\{k < \tau: X_k=v\}+1$, so 
\[
\pi(v)  = \P_s(X_\tau = v) = \E_s[\#\{k \le \tau: X_k=v\}] - \E_s[\# \{k < \tau: X_k=v\}]. 
\]
If $v \ne s$ then by the Markov property, 
\begin{align*}
\E_s [\#\{k \le \tau: X_k=v\} ]
& = \sum_{u} \E_s [\#\{k < \tau: X_k=u,X_{k+1}=v\}] \\
& = \sum_{u} y_u \pi(u) p_{uv} = \sum_{u} y_u Q(u,v)\, .
\end{align*}
Since 
\[
\E_s[\# \{k < \tau: X_k=v\}] = y_v \pi(v) = \sum_u y_v \pi(v)p_{vu}= \sum_u y_v Q(v,u)\, ,
\]
we obtain $\pi(v) = \sum_u (y_u Q(u,v) - y_v Q(v,u))$. 
For $Z \subset V$ with $s \not \in Z$, we thus have 
\[
\pi(Z) = \sum_{u\in V} \sum_{v \in Z} (y_u Q(u,v) - y_v Q(v,u))
\]
and the result follows since, by symmetry, 
$\sum_{u,v\in Z} (y_u Q(u,v) - y_v Q(v,u))=0$. 
\end{proof}

\begin{cor}\label{cor:connected}
For any $i \in V$, the set $\calA(i):=\{j\in V:y_j \ge y_i\}$ is internally connected from $s$: that is, for every vertex $v$ in $\calA(i)$, there exists a path from $s$ to $v$ within $\calA(i)$.
\end{cor}

\begin{proof}
Let $Z$ be the set of vertices $v$ in $\calA(i)$ such that there does not exist a path from $s$ to $v$ within $\calA(i)$. For any $k \in Z$ and $l \not \in Z$, if $Q(l,k) > 0$ then necessarily $l \not \in \calA(i)$ (otherwise there would be a path from $s$ to $k$ within $\calA(i)$) so we have $y_l < y_i$. Also, for any $k\in Z$ we have $y_k\ge y_i$. Applying Lemma \ref{lem:pizsum}, since $s \not \in Z$,
\[
\pi(Z) = \sum_{l \not \in Z} \sum_{k \in Z} (y_l Q(l,k) - y_k Q(k,l))
\le y_i Q(Z^c,Z) - y_i Q(Z,Z^c)\, ,
\]
with equality if and only $Z=\emptyset$. But \eqref{eq:Qswap} tells us that $Q(Z^c,Z) = Q(Z,Z^c)$, so the right-hand side is zero and the result follows. 
\end{proof}

It follows that $y_s$ is at least as large as any of the other $y_v$, so $y_s=y_N$ and we may assume that $s=N$. 
(Recall that we also have $h=1$ where $h$ is the halting state, and $y_1=0$.)

For $i=1,\ldots,N$, let $B_i$ be the set of vertices in $\{1,\ldots,i\}$ that are internally connected to $1$, i.e.~the set of $v\in\{1,\ldots,i\}$ such that there exists a path from $v$ to $1$ within $\{1,\ldots,i\}$.

\begin{lem}\label{ytophi}
For all $1 \le i \le j < N$, 
\[
y_{j+1}-y_i \le \frac{\pi(B_i)}{Q(B_j^c,B_i)}\, .
\]
\end{lem}

\begin{proof}
Since for $i<N$ we know that $s \not \in B_i$, by Lemma \ref{lem:pizsum} we have 
\begin{align*}
\pi(B_i) 
& = \sum_{l \not \in B_i} \sum_{k \in B_i} 
(y_l Q(l,k) - y_k Q(k,l)) \\
& \ge \sum_{l \not \in B_i} \sum_{k \in B_i}
y_l Q(l,k) - y_i Q(B_i,B_i^c) \\
& = \sum_{l \not \in B_j} \sum_{k \in B_i}
y_l Q(l,k)
+\sum_{l \in B_j\setminus B_i} \sum_{k \in B_i}
y_l Q(l,k) 
- y_i Q(B_i,B_i^c).
\end{align*}
If $l \in B_j\setminus B_i$ and $Q(l,k)>0$ for some $k \in B_i$ then $y_l \ge y_i$, so the preceding bound gives 
\[
\pi(B_i) \ge 
\sum_{l \not \in B_j} \sum_{k \in B_i}
y_l Q(l,k)
+ y_iQ(B_j\setminus B_i,B_i) - y_i Q(B_i,B_i^c).
\]
Since (see \eqref{eq:Qswap}) $Q(A,A^c)=Q(A^c,A)$ for any $A$, we get
\[
\pi(B_i) \ge \sum_{l \not\in B_j}\sum_{k\in B_i} y_l Q(l,k) - y_i Q(B_j^c,B_i).
\]
Finally, if $l\not\in B_j$, $k\in B_i$, and $Q(l,k)>0$, then $y_l\ge y_{j+1}$. Therefore
\[\pi(B_i) \ge (y_{j+1}-y_i) Q(B_j^c,B_i)\, . \qedhere\]
\end{proof}

We can now prove our main result for this section.

\begin{proof}[Proof of Proposition \ref{upperbd}]
Fix $\theta\in(0,1)$. Let $m_1 = 1$ and for $i \ge 1$, define 
\[
m_{i+1} = \min\{m > m_{i}: Q(B_m^c,B_{m_i}) \le (1-\theta)Q(B_{m_i}^c,B_{m_i})\}\, ,
\]
or $m_{i+1}=N$ if no such $m$ exists. Let $l = \min\{i : m_i = N\}-1$. 
Note that for each $j$, by definition $B_j$ is internally connected to $1$ within $\{1,\ldots,j\}$. Since our Markov chain is irreducible, for any $i\le j$ such that $i\not\in B_j$, there must exist a path within $B_j^c$ from $i$ to $k$ for some $k>j$. Since $\{j+1,\ldots,N\}$ is internally connected from $N$ by Corollary \ref{cor:connected}, it follows that $B_j^c$ must be connected. Also, for any $i\leq l$, by the definition of $m_{i+1}$ we have 
\[(1-\theta)Q(B_{m_i}^c,B_{m_i}) \geq Q(B_{m_{i+1}}^c,B_{m_i}) = Q(B_{m_i}^c,B_{m_i}) - Q(B_{m_{i+1}}\setminus B_{m_i},B_{m_i})\]
and rearranging we get
\[Q(B_{m_{i+1}}\setminus B_{m_i},B_{m_i}) \ge \theta Q(B_{m_i}^c,B_{m_i});\]
thus $(B_{m_i},1 \le i \le l)$ is a $\theta$-bottleneck sequence.

Now, since $y_1=0$, we have
\begin{align*}
\E_{s}[\tau]	= \sum_{v=1}^N \pi(v) y_v &\le \sum_{v=1}^N \pi(v) \cdot \sum_{j=1}^{l} (y_{m_{j+1}}-y_{m_{j}})\ind_{\{v > m_{j}\}} \\
& = 
\sum_{j=1}^{l} (y_{m_{j+1}}-y_{m_j}) \pi(\{m_j+1,\ldots,N\}) \\
& \le 
\sum_{j=1}^{l} (y_{m_{j+1}}-y_{m_j}) \pi(B_{m_j}^c) \\
& \le
\sum_{j=1}^{l} \frac{\pi(B_{m_j}) \pi(B_{m_j}^c)}{Q(B_{m_{j+1}-1}^c,B_{m_j})}\, ,
\end{align*}
the last inequality by Lemma~\ref{ytophi}.

By the definition of $m_{j+1}$, we have $Q(B_{m_{j+1}-1}^c,B_{m_j}) > (1-\theta) Q(B_{m_j}^c,B_{m_j})$, so the preceding inequality yields 
\[
\E_s[\tau] < \frac{1}{1-\theta} \sum_{j=1}^l \frac{1}{\Phi(B_{m_j})}\, . \qedhere
\]
\end{proof}

This completes the proof of Theorem \ref{upperbd1_new}. We now note that if our graph is a tree, then any connected set $A\subset V$ with connected complement has a single edge between $A$ and $A^c$, and therefore every $\theta$-bottleneck sequence (for $\theta\in(0,1]$) is a $1$-bottleneck sequence. We also recall that the same property implies that every Markov chain on a tree is reversible, and we may therefore unambiguously set $c_{uv}=\pi(u)p_{uv}$ for each $uv\in V$. Combining Theorem \ref{upperbd1_new} with Proposition \ref{treelower}, we obtain:

\begin{cor}\label{treecor}
Suppose that $X$ is a lazy Markov chain on a tree. Then
\[\tmix(X) \asymp \max_{(S_1,\ldots,S_l)\in \Sc_1(T)}\sum_{i=1}^l \frac{1}{\Phi(S_l)}.\]
Therefore if $X$ and $Y$ are two lazy Markov chains on the same tree, with conductances satisfying $\eps c^Y_{uv} \le c^X_{uv} \le c^Y_{uv}/\eps$ for all $uv\in E$ and some $\eps>0$, we have
\[\tmix(X)\asymp_\eps \tmix(Y).\]
\end{cor}

The latter part of this result was originally proved by Peres and Sousi \cite{peres:mixing} using very different methods. Hermon and Peres \cite{hermon_peres:characterization} also gave a similar result for the $L^2$ mixing time.

\section{Lower bounds for the mixing time on more general graphs}\label{gen_lb_sec}

Proposition \ref{treelower} gives a lower bound for the mixing time on trees in terms of $1$-bottleneck sequences. The proof uses the tree structure in a non-trivial way, but the idea, that in order to mix we must be able to hit every large set, and to hit a large set we must travel through a sequence of bottlenecks, holds more generally. Indeed, if $G$ is roughly isometric to a tree $T$, then a similar approach can be used to give a lower bound on the mixing time.

We recall some definitions: for any graph $G$, $d_G$ denotes the graph distance on $V$. For $A\subset V$ and $r\ge0$ we write $B_G(r,A) = \{v\in V : d_G(v,A)\leq r\}$, and for $A,B,C\subset V$ we write $A\lr{C}B$ if every path between $A$ and $B$ contains a vertex of $C$.

\begin{prop}\label{graph_tree_lb}
Fix a finite connected graph $G=(V(G),E(G))$ and a finite tree $T=(V(T),E(T))$. Suppose that $X$ is an $\eps$-uniform Markov chain on $G$, and $Y$ is an $\eps$-uniform Markov chain on $T$. Suppose also that $G \simeq_r T$, and $G$ and $T$ both have maximum degree at most $\Delta$. Then there exists $\eta>0$, depending only on $\Delta$, $\cond$ and $r$, such that
\[\thit(\eta,X) \gtrsim_{\Delta,\cond,r} \max_{(S_1,\ldots,S_l) \in \Sc_1(T)} \sum_{i=1}^{l} \frac{|S_i||S_i^c|}{|V(T)|} \asymp_{\Delta,\cond} \max_{(S_1,\ldots,S_l) \in \Sc_1(T)} \sum_{i=1}^{l} \frac{1}{\Phi^Y(S_i)}.\]
\end{prop}

To prove this, our strategy is as follows. Take a 1-bottleneck sequence for $T$. Between any set $S_j$ in that sequence and its complement $S_j^c$, there is a single edge. The same is not necessarily true for bottleneck sequences on $G$, but we can use the rough isometry between $T$ and $G$ to build a small subset of $V(G)$ that plays the role of the single edge in $T$, in that when this subset is removed from $G$ it splits the graph into at least two components. These two components will be of similar sizes to $S_j$ and $S_j^c$. In order to mix, the Markov chain will have to pass through all of the small subsets, and this will yield the required bound.

In order to make this heuristic precise, fix a correspondence $\cC \subset V(T) \times V(G)$ with $\mathrm{str}(\cC) \le r$. For $t \in V(T)$ write $\cC_{t,G} = \{v \in V(G): (t,v) \in \cC\}$, and likewise for $v \in V(G)$ let $\cC_{T,v} = \{t \in V(T): (t,v) \in \cC\}$.

Next, for $t \in V(T)$ let 
\begin{equation}\label{eq:atdef}
A_t = \{v \in V(G): d_G(v,\cC_{t,G}) \le r(r+1)\} = B_G(r(r+1),\cC_{t,G}).
\end{equation}

We will need some basic properties of the sets $A_t$ before we can proceed.

\begin{lem}\label{lem:atconnect}
For all $t \in V(T)$, $A_t$ is connected. Moreover, for all $u,v \in A_t$, there is a path $P$ from $u$ to $v$ contained within $A_t$, of length at most $2r(r+1)+(r-1)$. 
\end{lem}

\begin{proof}
If $u,v \in \cC_{t,G}$ then $d_G(u,v) \le r-1$ since $\str(\cC) \le r$. It follows that $B_G(\lceil (r-1)/2\rceil,\cC_{t,G})$ is connected, so $A_t$ is also connected. To prove the second claim, note that for any $u,v \in A_t$, $d_G(u,v) \le d_G(u,\cC_{t,G}) + d_G(v,\cC_{t,G}) + r-1$, and that any shortest path between elements of $\cC_{t,G}$ is contained within $A_t$. 
\end{proof}

The next lemma relates the components of $A_t^c$ (in $G$) with those of $V(T)\setminus \{t\}$ (in $T$).
For $x,y \in V(T)$ write $\bbr{x,y}$ for the unique path between $x$ and $y$ in $T$. 

\begin{lem}\label{lem:cut_off}
For all distinct $a,b,t \in V(T)$, if $t \in \bbr{a,b}$ then $\cC_{a,G} \lr{A_t} \cC_{b,G}$.
\end{lem}

\begin{proof}
Fix distinct $a,b,t \in V(T)$ with $t \in \bbr{a,b}$, and $x \in \cC_{a,G}$ and $y \in \cC_{b,G}$, and consider any path $x=x_0,x_1,\ldots,x_l=y$ between $x$ and $y$ in $G$. 

For each $i < l$, if $a' \in \cC_{T,x_i}$ and $b' \in \cC_{T,x_{i+1}}$ then since $\str(\cC) \le r$ we have 
\[
d_T(a',b')+1 \le r(d_G(x_i,x_{i+1})+1) = 2r. 
\]
It follows that $\bigcup_{i \le l} B_T(r,\cC_{T,x_i})$ contains a path from $a$ to $b$. Since $T$ is a tree, any such path contains $t$, so this implies there is $i \le l $ and $s \in \cC_{T,x_i}$ with $d_T(s,t) \le r$. 

Finally, fix any vertex $v \in \cC_{t,G}$. Then since $(s,x_i) \in \cC$ and $(t,v) \in \cC$, again using that $\str(\cC) \le r$, we obtain that 
\[
d_G(x_i,v)+1 \le r(d_T(s,t)+1) \le r(r+1),
\]
so $x_i \in A_t$. Since the path between $x$ and $y$ was arbitrary, it follows that $x \stackrel{A_t}{\longleftrightarrow} y$.
\end{proof}

\begin{lem}\label{lem:atdisj}
For any $s,t \in V(T)$, if $d_T(s,t) \ge 2r^2(r+1)+r$ then $A_s \cap A_t = \emptyset$.
\end{lem}

\begin{proof}
If $x \in \cC_{s,G}$ and $y \in \cC_{t,G}$ 
then $d_G(x,y)+1 \ge (d_T(s,t)+1)/r > 2r(r+1)+1$. The lemma follows. 
\end{proof}

A large part of our proof of Proposition \ref{graph_tree_lb} essentially boils down to invoking Kac's formula. See for example \cite[Corollary 2.24]{aldous:reversible} for a standard form of this result; however, we need to apply it in a non-standard way, and so the alternative form below will be useful.

\begin{lem}[Kac's formula]\label{kac}
Suppose that $X$ is a finite irreducible Markov chain with state space $V$. If $L$, $C$ and $R$ partition $V$ and $L \lr{C} R$ then
\[\pi(L\cup C) = \pi(C)\E_{\pi_C}[H^+(L^c)]\]
where $\pi_C$ is the probability measure on $V$ given by $\pi_C(v) = (\pi(v)/\pi(C)) \I{v \in C}$. 
\end{lem}

\begin{proof}
If $k\geq 2$, then
\begin{align*}
&\P_\pi(H^+(C)=k, X_1\in L\cup C)\\
&= \P_\pi(X_1\in L,\ldots, X_{k-1}\in L, X_k\in C)\\
&= \P_\pi(X_1\in L,\ldots, X_{k-1}\in L) - \P_\pi(X_1\in L,\ldots, X_k\in L)\\
&= \P_\pi(X_1\in L,\ldots, X_{k-1}\in L) - \P_\pi(X_0\in L,\ldots, X_{k-1}\in L)\\
&= \P_\pi(X_0\in C, X_1\in L,\ldots, X_{k-1}\in L)\\
&= \pi(C)\P_{\pi_C}(H^+(L^c)\geq k).
\end{align*}
where we used stationarity for the third equality. If $k=1$, then
\[\P_\pi(H^+(C)=k, X_1\in L\cup C) = \P_\pi(X_1\in C) = \pi(C) = \pi(C)\P_{\pi_C}(H^+(L^c)\geq 1).\]
Summing over $k$, we get
\[\pi(L\cup C) = \P_\pi(X_1\in L\cup C) = \pi(C)\E_{\pi_C}[H^+(L^c)].\qedhere\]
\end{proof}

Now fix a $1$-bottleneck sequence $(S_1,\ldots,S_l)$ for $T$, and let $M=\max\{i : \pi(S_i)\le 1/2\}$. For each $i=1,\ldots,M$, let $t_i$ be the unique vertex that is in $S_i$ and has a neighbour in $S_i^c$. Let $m_1 = M$, and for $j\ge 2$, define
\[m_j = \max\{m : d_T(t_m, t_{m_{j-1}}) \ge 2r^2(r+1)+r\}\]
or $m_j=0$ if no such $m$ exists. Let $K=\max\{j : m_j\ge 1\}$.

For $j=1,\ldots,K$, let $n_j = m_{K-j+1}$, and define
\[C_j = A_{t_{n_j}}, \quad L_j = \bigcup_{t\in S_{n_j}} \cC_{t,G} \setminus A_{t_{n_j}}, \quad R_j = V\setminus(L_j\cup C_j) = \bigcup_{t\in S_{n_j}^c} \cC_{t,G} \setminus A_{t_{n_j}}. \]

The next lemma does most of the work in proving our main result for this section, Proposition \ref{graph_tree_lb}.

\begin{lem}\label{lem:hitR}
If $v\in C_1$, then
\[\E_v[H(R_K)] \gtrsim_{\Delta,\cond,r} \sum_{i=1}^M |S_i|.\]
\end{lem}

\begin{proof}
By Lemma \ref{lem:atdisj}, we know that the sets $C_1,\ldots,C_K$ are pairwise disjoint. By Lemma \ref{lem:cut_off} we know that for each $j$, $L_j\lr{C_j} R_j$. In particular we have $C_{j-1}\subset L_j$ for each $j=2,\ldots,K$, and $C_{j+1}\subset R_j$ for each $j=1,\ldots K-1$.

Therefore if $v\in C_1$,
\begin{align*}
\E_v[H(R_K)] &\ge \E_v[H(R_1)] + \min_{u\in C_2} \E_u[H(R_2)] + \ldots + \min_{u\in C_L} \E_u[H(R_K)]\\
&\ge \sum_{j=1}^K \min_{u\in C_j} \E_u[H(R_j)].
\end{align*}
Now, if $u\in C_j$, by Lemma \ref{lem:atconnect} we have
\[\E_u[H(R_j)] \ge \E_u[H^+(L_j^c)] \asymp_{\Delta,\cond,r} \E_{\pi_{C_j}}[H^+(L_j^c)].\]
By Lemma \ref{kac}, this equals $\pi(L_j\cup C_j)/\pi(C_j)$, and applying Lemma \ref{lem:atconnect} again we get
\[\E_v[H(R_K)] \gtrsim_{\Delta,\cond,r} \sum_{j=1}^K \frac{\pi(L_j\cup C_j)}{\pi(C_j)} \asymp_{\Delta,\cond,r} \sum_{j=1}^K |L_j\cup C_j| \asymp_{\Delta,\cond,r} \sum_{j=1}^K |S_{n_j}|.\]
Finally, we have $n_{j+1}\le n_j + 2r^2(r+1)+r$ for each $j$. We also have $|S_i|\le |S_j|$ for all $i\le j$. Therefore $\sum_{i=1}^M |S_i|\le (2r^2(r+1)+r)\sum_{j=1}^K |S_{n_j}|$, which completes the proof.
\end{proof}

\begin{proof}[Proof of Proposition \ref{graph_tree_lb}]
Since $R_K = \bigcup_{t\in S_M^c} \cC_{t,G} \setminus A_{t_M}$ and we know that $\pi(S_M^c)\ge 1/2$, assuming that $G$ and $T$ are both large (otherwise the result is trivial) by Lemma \ref{lem:atconnect} we have $\pi(R_K)\gtrsim_{\Delta,\cond,r} 1$. Thus there exists $\eta>0$ depending only on $\Delta$, $\cond$ and $r$ such that
\[\thit(\eta,X) \geq \max_{v\in C_1}\E[H(R_K)].\]

Applying Lemma \ref{lem:hitR} we get
\[\thit(\eta,X) \gtrsim_{\Delta,\cond,r} \sum_{i=1}^M |S_i|\ge \sum_{i=1}^M \frac{|S_i||S_i^c|}{|V(T)|}.\]
Since $S_l^c,S_{l-1}^c,\ldots, S_1^c$ is also a $1$-bottleneck sequence, applying the same argument to this sequence gives $\thit(\eta,X) \gtrsim_{\Delta,\cond,r} \sum_{i=M+1}^l \frac{|S_i||S_i^c|}{|V(T)|}$. The result follows.
\end{proof}

\section{The bottleneck sequence game: proof of Theorem \ref{upperbd2}}\label{gameproofsec}

We recall the bottleneck sequence game from Section \ref{gameintro}, and aim to prove Theorem \ref{upperbd2}. First we check that both players can always move. 
\begin{lem}\label{welldef_new}
Fix $\near,\adjust \in (0,1]$, $\isop \in (0,1)$ and $s \in V$. Suppose that $((C_i,D_i),0 \le i \le k)$ is an $(s,\near,\adjust,\isop)$-valid sequence and $D_i \ne V$. 
Then \Pone~has at least one $\isop$-valid move $(C_{k+1},D_k)$ from position $(C_k,D_k)$. Moreover, from any such position $(C_{k+1},D_k)$, \Ptwo~has at least one $(s,\near,\adjust)$-valid move $(C_{k+1},D_{k+1})$. 
\end{lem}
\begin{proof}
Fix $((C_i,D_i),0 \le i \le k)$ as in the statement of the lemma. We begin by showing that \Pone~has a valid move. More specifically, writing $C'=C_k \cup D_k^c$, we show that $(C',D_k)$ is a $\isop$-valid move for \Pone~starting from $(C_k,D_k)$

If $k=0$ then $C_k=D_k=\emptyset$ and $C'=D_k^c=V$, in which case it is easy to see that $(C',D_k)$ is a valid move for \Pone. If $k>0$ and $D_k\neq V$ then first note that $C_k$ and $D_k^c$ are both connected since $((C_i,D_i),0 \le i \le k)$ is a valid sequence. Moreover, using the adjustment property for $D_k$, with $S=\emptyset$, we have that $Q(D_k^c,C_k) \ge \beta Q(D_k^c,D_k) > 0$, so $C'$ is connected. Next, $Q((D_k \cup C')^c,C_k) = Q(\emptyset,C_k) = 0$, so $C'$ satisfies the isoperimetry requirement. Thus $(C',D_k)$ is also a valid move for \Pone~in the case $k > 0$. 

Now let $(C_{k+1},D_k)$ be any $\isop$-valid move for \Pone~starting from $(C_k,D_k)$, and let $D' = h_s(C_{k+1} \cup D_k)$, the set of all vertices separated from $s$ by $C_{k+1} \cup D_k$. We claim that $(C_{k+1},D')$ is a $(s,\near,\adjust)$-valid move for \Ptwo~from position $(C_{k+1},D_k)$. 

First, since $D_k \ne V$, the endgame property (iv) applied to $(C_k,D_k)$ yields that $s \not \in D_k$. It follows that $s \in D'$ if and only if $s \in C_{k+1}$. In this case $D' = h_s(C_{k+1} \cup D_k)=V(G)$, which verifies (iv). In this case (i) is obvious, and (ii) and (iii) are vacuously true, so $(C_{k+1},D')$ is a valid move in this case. 

We henceforth assume that $s \not \in C_{k+1}$ so $s \not \in D'$, and (iv) is therefore satisfied. In this case, by definition $(D')^c$ is the connected component of $(C_{k+1} \cup D_k)^c$ containing $s$. In particular, $(D')^c$ is connected so (i) is satisfied. 

In what follows it is convenient to write $R = D' \setminus (C_{k+1} \cup D_k)$, so $R$ consists of all connected components of $(C_{k+1} \cup D_k)^c$ except the one containing $s$. Note that there are no edges between $R$ and $(D')^c$.

Fix $v \in \inbd D'$ and a neighbour $w$ of $v$ with $w \in (D')^c$. By the observation of the preceding paragraph, $v \not \in R$ so either $v \in C_{k+1}$ or $v \in D_k$. If $v \in C_{k+1}$ then $v$ is certainly $\near$-near to $C_{k+1}$. If $v \in D_k$ then since $w \in D_k^c$ we have $v \in \inbd D_k$. 
By property (ii) applied to $(C_k,D_k)$ we obtain that $v$ is $\alpha$-near to $C_k$ and thus to $C_{k+1}$. This establishes (ii). 

It remains to show that $D'$ is a $\adjust$-adjustment of $C_{k+1}$. For this fix any set $S \subset (D')^c$ for which $C_{k+1} \cup S$ is connected, and let $S' = (D'\setminus D_k) \cup S$. 

We wish to apply (iii) to $(C_k,D_k)$ and $S'$, and for this we need that $C_k \cup S'$ is connected. To establish this, first note that $(C_{k+1}\setminus C_k) \subset D_k^c$, so $D' \setminus D_k = (C_{k+1}\setminus C_k) \cup R$. It follows that 
\[
C_k \cup S' = C_k \cup (C_{k+1}\setminus C_k) \cup R \cup S = C_{k+1} \cup R \cup S \, .
\]
Since $D_k^c$ is connected by (i), every component of $(C_{k+1} \cup D_k)^c$ must have an edge between it and $C_{k+1}$. It follows that $C_{k+1} \cup R$ is connected. But $C_{k+1} \cup S$ is connected by assumption, so $C_k \cup S'=C_{k+1} \cup R \cup S$ is indeed connected. 

Applying (iii) to $(C_k,D_k)$ and $S'$, we now obtain that 
\[
Q((D_k \cup S')^c, C_k \cup S') \ge \adjust Q((D_k \cup S')^c, D_k \cup S')\,  .
\]
Noting that $D_k \cup S' = D' \cup S$ and $C_k \cup S'=C_{k+1} \cup R \cup S$, this inequality becomes
\[
Q((D' \cup S)^c, C_{k+1} \cup R \cup S) \ge \adjust Q((D' \cup S)^c, D' \cup S)\, .
\]
But $R$ consists of all components of $(C_{k+1} \cup D_k)^c$ except the one containing $s$, and $(D')^c$ is the component of $(C_{k+1} \cup D_k)^c$ containing $s$, so $Q((D' \cup S)^c, R) \le Q((D')^c,R)=0$. We thus have 
\[
Q((D' \cup S)^c, C_{k+1} \cup S) \ge \adjust Q((D' \cup S)^c, D' \cup S)\, ,
\]
so $D'$ is a $\adjust$-adjustment of $C_{k+1}$. This establishes (iii) for $D'$ and completes the proof. \end{proof}

At this point we recall the scaled exit frequencies $(y_v,\, v\in V)$ from Section \ref{t2proof}, and observe the following consequence of nearness.

\begin{lem}\label{nearlem}
Suppose that $\tau$ is a stopping rule from $s$ to $\pi$. For any $u,v\in V$ and $m \in \mathbb{N}$, if $v$ is $1/m$-near to $u$ then
\[
y_v \le m(m+1)y_u + m^2(m+1)/2\, .
\]
\end{lem}

\begin{proof}
For any $u,v\in v$ we have
\begin{align*}
&\E_s[\#\{j<\tau : X_j = u\}]\\
&\ge \E_s[\#\{j < \tau-k : X_j=v,\, X_{j+k}=u\}]\\
&=\sum_{j=0}^\infty \P_s(X_j=v,\, X_{j+k}=u,\, \tau>j+k)\\
&=\sum_{j=0}^\infty \P_s(X_j=v,\, X_{j+k}=u,\, \tau>j) - \sum_{j=0}^\infty \P_s(X_j=v,\, X_{j+k}=u,\, j<\tau\le j+k).
\end{align*}
By the simple Markov property,
\[\P_s(X_j=v,\, X_{j+k}=u,\, \tau>j) = \P_s(X_j=v,\, \tau>j)\P_v(X_k=u),\]
so
\begin{align*}
&\E_s[\#\{j<\tau : X_j = u\}]\\
&\ge \sum_{j=0}^\infty \P_s(X_j=v,\, \tau>j)\P_v(X_k=u) - \sum_{j=0}^\infty \P_s(X_{j+k}=u,\, j<\tau\le j+k)\\
&\ge \E_s[\#\{j < \tau : X_j=v\}]\P_v(X_k=u) - \E_s[\#\{ i \in\{\tau,\ldots,\tau+k-1\} : X_i=u\}]\\
&= \E_s[\#\{j < \tau : X_j=v\}] \P_v(X_k=u) - \sum_{j=0}^{k-1}\P_s(X_{\tau+j}=u).
\end{align*}
But $X_{\tau+j}$ has distribution $\pi$ for all $j \ge 0$, so recalling that $\E_s[\#\{j<\tau : X_j = u\}]=\pi(u)y_u$, and similarly for $v$, we get
\[\pi(u)y_u \ge \pi(v)y_v \P_v(X_k=u) - k\pi(u).\]
Finally, for any $m \in \mathbb{N}$ we have  $\P_v(H(u) \le m) \le \sum_{k=0}^m \P_v(X_k=u)$, so summing the preceding bound over $k \le m$ gives 
\[
(m+1)\pi(u)y_u \ge \pi(v)y_v \P_v(H(u) \le m) - m(m+1)\pi(u)/2.
\]
Dividing through by $\pi(u)$, we have
\[(m+1)y_u \ge \frac{\pi(v)y_v}{\pi(u)}\cdot \P_v(H(u)\le m) - m(m+1)/2.\]
If $v$ is $(1/m)$-near to $u$ we have $\pi(v)\P_v(H(u) \le m) \ge \pi(u)/m$, and the result follows.
\end{proof}

In order to prove Theorem \ref{upperbd2}, by (\ref{equiv2}), it is enough to prove the following.

\begin{prop}\label{betterupperbd}
Fix a graph $G=(V,E)$, and suppose that $X$ is a Markov chain on $G$. Take any $s\in V$, $\near,\adjust\in(0,1]$ such that $1/\near\in\mathbb{N}$, and $\isop\in (0,1)$. If $\tau$ is an optimal stopping rule from $s$ to $\pi$, then there exist $\isop$-valid moves for \Pone~such that, whatever $(s,\near,\adjust)$-valid moves \Ptwo~makes,
\begin{equation}\label{strategybd}
\E_s[\tau] \leq \frac{1}{\near^3} + \frac{2}{\near^2\adjust\isop}\sum_{k=1}^l \frac{1}{\Phi(D_k)}.
\end{equation}
\end{prop}

\begin{proof}
We proceed similarly to the proof of Theorem \ref{upperbd1_new}. Fix a starting state $s$ and $\near,\adjust\in(0,1]$, $\isop\in (0,1)$. We describe a strategy for \Pone~for which, whatever the strategy of \Ptwo, any resulting sequence $D_1,\ldots,D_l$ satisfies (\ref{strategybd}).

Label the vertices of $V(G)$ as $\{1,\ldots, N\}$ such that $y_1\leq y_2\leq\ldots\leq y_N$. 
By Corollary \ref{cor:connected} we may assume that $s=N$.

Let $v_1=1$ and note that $v_1$ is a halting vertex, i.e.~$y_1=0$ (since there is always at least one such vertex \cite{lovasz_winkler:efficient_stopping_rules}). Set $C_1=\{v_1\}$. \Ptwo~can then choose any $D_1$ that satisfies rules (i) to (iv).

Given $C_{k}$, $D_{k}$ and $v_{k}$, if $D_k=V$ then we stop and set $l=k-1$; otherwise we define $C_{k+1}$ and $v_{k+1}$ as follows.

List the vertices of $D_k^c$ as $v_1^k,\ldots,v_{n_k}^k$ such that $v_1^k \leq \ldots \leq v_{n_k}^k$. Let $C_j^k$ be the set of vertices of $C_k\cup\{v_1^k,\ldots,v_j^k\}$ that are internally connected to $v_1$.

Choose $m_{k+1} = \min\{m : Q((D_k\cup C_m^k)^c,C_k) \leq \isop Q(D_k^c,C_k)\}$, or $m_{k+1}=N$ if no such $m$ exists. Let $v_{k+1}= v^k_{m_{k+1}}$ and $C_{k+1}=C_{m_{k+1}}^k$, and $l=\min\{k : m_k=N\}-1$. Note that $v_{k+1} = \max\{v\in C_{k+1}\}$ (we have $v\geq v_k$ for all $v\in \partial C_k$, otherwise $v$ would have been added to $C_k$ at the previous step; so some vertex larger than $v_k$ is added to $C_k$ in constructing $C_{k+1}$; and $v_{k+1}$ is the largest vertex added).

By construction, the sequence $(C_1,\ldots,C_l)$ satisfies conditions (a) and (b) in the bottleneck sequence game and is therefore a valid strategy. We now show that any sequence $D_1,\ldots,D_l$ that can arise when \Pone~follows this strategy satisfies (\ref{strategybd}).

Let $U_k$ be the subset of $\{u\in V : u\leq v_k\}$ that is internally connected to $v_1$. Note that by our construction of $C_k$, we must have $C_k\subset U_k\subset D_{k-1}\cup C_k$. For $k\le l$, $s\not\in U_k$, so by Lemma \ref{lem:pizsum}
\begin{align*}
\pi(U_k) &= \sum_{u\in U_k,v\not\in U_k} (y_v Q(v,u) - y_u Q(u,v))\\
&\ge \sum_{u\in U_k,v\not\in U_k} y_v Q(v,u) - y_{v_k} Q(U_k,U_k^c)\\
&=\sum_{u\in U_k,v\not\in U_k} y_v Q(v,u) - y_{v_k} Q(U_k^c,U_k)\\
&=\sum_{u\in U_k,v\not\in U_k} (y_v- y_{v_k}) Q(v,u)
\end{align*}
where we used \eqref{eq:Qswap} for the penultimate equality. Let $Z_k = (D_k\cup C_{k+1})^c\cup\{v_{k+1}\}$. If $v\not\in U_k$ and $\exists u\in U_k$ such that $Q(v,u)>0$, then $y_v\ge y_{v_k}$; thus, since $C_k\subset U_k$ and $Z_k\subset U_k^c$,
\[\pi(U_k) \ge \sum_{u\in C_k,v\in Z_k} (y_v-y_{v_k})Q(v,u).\]
Also note that if $v\in Z_k$ and $\exists u\in U_k$ such that $Q(u,v)>0$, then $y_v\ge y_{v_{k+1}}$. Thus
\begin{equation}\label{Uk_est}
\pi(U_k) \ge \sum_{u\in C_k,v\in Z_k} (y_{v_{k+1}}-y_{v_k})Q(v,u) = (y_{v_{k+1}}-y_{v_k})Q(Z_k,C_k).
\end{equation}

By our construction of $C_{k+1}$, we have
\[Q(Z_k,C_k) = Q((D_k\cup C_{k+1})^c\cup\{v_{k+1}\},C_k) > \isop Q(D_k^c,C_k),\]
and by rule (iii) of the bottleneck sequence game applied to $(C_k,D_k)$, with $S=\emptyset$, we have $Q(D_k^c,C_k) \geq \adjust Q(D_k^c,D_k)$, so
\[Q(Z_k,C_k) > \adjust\isop Q(D_k^c,D_k).\]
But $U_k\subset D_{k-1}\cup C_k\subset D_k$ (by rule (i) of the bottleneck sequence game), so $\pi(D_k)\geq \pi(U_k)$. Substituting these estimates back into (\ref{Uk_est}), we get
\[\pi(D_k) > \beta\gamma (y_{v_{k+1}} - y_{v_k}) Q(D_k^c,D_k).\]
Multiplying through by $\pi(D_k^c)$ and rearranging, we obtain the key estimate
\begin{equation}\label{yeq}
\pi(D_k^c)(y_{v_{k+1}} - y_{v_k}) < \frac{\pi(D_k^c)\pi(D_k)}{\adjust\isop Q(D_k^c,D_k)} = \frac{1}{\adjust\isop \Phi(D_k)}.
\end{equation}

If $k\leq l$ and $v\in \inbd D_k$, then by rule (ii), $v$ is $\near$-near to $C_k$, so by Lemma \ref{nearlem} we have
\[y_v \leq \frac{2}{\near^2}\max_{u\in C_k} y_u + \frac{1}{\near^3}.\]

Suppose $k\leq l$ and take $w\in D_k$. Since, by Corollary \ref{cor:connected}, $\{v : y_v \geq y_w\}$ is a connected set containing both $s$ and $w$, and $s\in D_k^c$, there must be some vertex $v\in \inbd D_k$ such that $y_v\geq y_w$. Therefore
\begin{equation}\label{maxeq}
\max_{w\in D_k} y_w \leq \max_{v\in \inbd D_k} y_v \leq \frac{2}{\near^2}\max_{u\in C_k} y_u + \frac{1}{\near^3} = \frac{2}{\near^2} y_{v_k} + \frac{1}{\near^3}.
\end{equation}
Similarly, using rule (iv) and Lemma \ref{nearlem},
\begin{equation}\label{lastmaxeq}
\max_{w\in D_{l+1}} y_w \leq y_s \leq \frac{2}{\near^2} \max_{u\in C_{l+1}} y_u + \frac{1}{\near^3} = \frac{2}{\near^2} y_{v_{l+1}} + \frac{1}{\near^3}.
\end{equation}

Define $q:V\to\mathbb{N}$ by setting $q(v) = k$ if $v\in D_k\setminus D_{k-1}$. Then by (\ref{maxeq}) and (\ref{lastmaxeq}),
\[\E_s[\tau] = \sum_{u\in V} \pi(u)y_u \leq \sum_{u\in V}\pi(u)\max_{w\in D_{q(u)}}y_w \leq \frac{1}{\near^3} + \frac{2}{\near^2}\sum_{u\in V}\pi(u)y_{v_{q(u)}}.\]
But since $y_{v_1}=y_1=0$, we have
\begin{align*}
\sum_{u\in V}\pi(u)y_{v_{q(u)}} &= \sum_{u\in V} \pi(u)\sum_{k=1}^{q(u)-1} (y_{v_{k+1}} - y_{v_k})\\
&\leq \sum_{u\in V} \pi(u) \sum_{k=1}^{l} (y_{v_{k+1}}-y_{v_k})\ind_{\{u\not\in D_k\}}\\
&= \sum_{k=1}^{l} \pi(D_k^c)(y_{v_{k+1}}-y_{v_k})
\end{align*}
where the inequality follows from the fact that if $k\leq q(u)-1$ then $u\not\in D_k$. By (\ref{yeq}), this is bounded above by $\frac{1}{\adjust\isop}\sum_{k=1}^{l}\frac{1}{\Phi(D_k)}$, and therefore
\[\E_s[\tau] \leq \frac{1}{\near^3} + \frac{2}{\near^2\adjust\isop}\sum_{k=1}^{l}\frac{1}{\Phi(D_k)}.\qedhere\]
\end{proof}

This completes the proof of Theorem \ref{upperbd2}.

\section{Robustness of mixing: proof of Theorem \ref{RIequiv_new}}\label{thm1proof}

Recall that a Markov chain $X$ on a graph $G=(V,E)$ is $\eps$-uniform if $\eps\pi(x)p_{xy} \le \pi(u)p_{uv}\le \pi(x)p_{xy}/\eps$ for all $uv,xy\in E$, and $p_{uv}=0$ for all $u,v\in V$ with $uv\not\in E$.

Our plan for proving Theorem \ref{RIequiv_new} is to apply Theorem \ref{upperbd2}, outlining valid moves for \Ptwo~that aim to keep $\sum_k 1/\Phi(D_k)$ small, no matter what valid moves \Pone~makes. The nature of rough isometry means that the bounds used in this section will necessarily be somewhat cruder than in previous sections. The following result will be enough to complete the proof.

\begin{prop}\label{robustprop}
Fix a finite connected graph $G=(V(G),E(G))$ and a finite tree $T=(V(T),E(T))$. Suppose that $X$ is an $\eps$-uniform Markov chain on $G$. Suppose also that $G \simeq_r T$, and $G$ and $T$ both have maximum degree at most $\Delta$. Then for any $s\in V$, there exist $\near,\adjust,\isop\in(0,1)$, depending only on $\Delta$, $\eps$ and $r$, and $(s,\near,\adjust)$-valid moves for \Ptwo~in the bottleneck sequence game, such that whatever $\isop$-valid moves \Pone~makes, the resulting sequence $D_1,\ldots,D_l$ satisfies
\begin{equation}\label{compare}
\sum_{n=1}^l |D_n||D_n^c| \lesssim_{\Delta,r} \max_{(S_1,\ldots,S_m)\in \Sc_1(T)} \sum_{j=1}^m |S_j||S_j^c|.
\end{equation}
\end{prop}

We delay the proof of Proposition \ref{robustprop} for a moment to check that it implies Theorem \ref{RIequiv_new}.

\begin{proof}[Proof of Theorem \ref{RIequiv_new}]
Fix any $\eps$-uniform Markov chain $Y$ on $T$. By \eqref{equiv2} it suffices to show that $\tstop(X)\asymp_{\Delta,\cond,r} \tstop(Y)$. Note first that since $G$ and $T$ are roughly isometric, we have $|V(G)|\asymp_{\Delta,r} |V(T)|$. It is also easy to check that for any $D\subset V(G)$, we have $\Phi^X(D) \gtrsim_{\Delta, \eps} \frac{|V(G)|}{|D||D^c|}$, and for any $S\subset V(T)$, $\Phi^Y(S)\asymp_{\Delta,\eps} \frac{|V(T)|}{|S||S^c|}$.

We now play the bottleneck sequence game, with \Pone~following the strategy from Theorem \ref{upperbd2} and \Ptwo~following the strategy from Proposition \ref{robustprop}. Then the resulting sequence $(D_1,\ldots,D_l)$ satisfies
\[\tstop(X)\lespar \sum_{k=1}^l \frac{1}{\Phi^X(D_k)} \lesssim_{\Delta,\eps} \sum_{k=1}^l \frac{|D_k||D_k^c|}{|V(G)|} \lesssim_{\Delta,r} \max_{(S_1,\ldots,S_m)\in \Sc_1(T)} \sum_{j=1}^m \frac{|S_j||S_j^c|}{|V(T)|}.\]
Also, by Proposition \ref{graph_tree_lb}, we have
\[\tstop(X)\gtrpar \max_{(S_1,\ldots,S_m)\in \Sc_1(T)} \sum_{j=1}^m \frac{|S_j||S_j^c|}{|V(T)|}.\]
But by Corollary \ref{treecor} we have
\[\tstop(Y)\eqpar \max_{(S_1,\ldots,S_m)\in \Sc_1(T)} \sum_{j=1}^m \frac{1}{\Phi^Y(S_j)} \asymp_{\Delta,\eps}\max_{(S_1,\ldots,S_m)\in \Sc_1(T)} \sum_{j=1}^m \frac{|S_j||S_j^c|}{|V(T)|}.\]
Therefore $\tstop(X)\eqpar \tstop(Y)$.
\end{proof}

In order to prove Proposition \ref{robustprop}, we need to describe a strategy for~\Ptwo. Let $R=2r^2-r-1$.

Given $s$ and~\Pone's first move $(C_1,\emptyset)$, let $\tilde C_1=B_G(R,C_1)$, the set of all vertices within distance $R$ of $C_1$. If $s\in \tilde C_1$ then let $D_1=V(G)$ and $l=0$ and stop. Otherwise choose a shortest path $\sigma^1 = (\sigma^1_0,\ldots,\sigma^1_{m_1})$ from $C_1$ to $s$. Note that $\sigma^1_i\in \tilde C_1$ for $i\leq R$ and $\sigma^1_i$ is in the connected component of $\tilde C_1^c$ containing $s$ for $i\geq R+1$. Let $D_1 = h_s(\tilde C_1)\setminus\{\sigma^1_1,\ldots,\sigma^1_{R}\}$.

We now recursively describe~\Ptwo's $n$th move for $n\ge2$; suppose that we are at position $(C_n,D_{n-1})$, that the previous move was $(C_{n-1},D_{n-1})$, and that we have a shortest path $\sigma^{n-1} = (\sigma^{n-1}_0,\ldots,\sigma^{n-1}_{m_{n-1}})$ from $C_{n-1}$ to $s$.

Let $\tilde C_n = B_G(R,C_n)$ and $\tilde C_{n-1} = B_G(R,C_{n-1})$. If $s\in \tilde C_n$ then set $D_n = V(G)$ and $l=n-1$ and stop. Otherwise let $I_n = \max\{i : \sigma^{n-1}_i\in C_n\}$, so that $\sigma_{I_n}^{n-1}$ is the last vertex in the path $\sigma^{n-1}$ within $C_n$. Note that (since $C_n$ is connected by rule (a)) if $I_n\le R-1$ then $C_n = C_{n-1}\cup\{\sigma^{n-1}_1,\ldots,\sigma^{n-1}_{I_n}\}$. In this case let $\sigma^n=(\sigma^{n-1}_{I_n},\ldots,\sigma^{n-1}_{m_{n-1}})$. On the other hand if $I_n\ge R$ then let $\sigma^n$ be any shortest path from $C_n$ to $s$. Finally let $D_n = h_s(\tilde C_n)\setminus\{\sigma^n_1,\ldots,\sigma^n_R\}$.

We want to show that this strategy is a valid one, and the first thing to check is that $D_n$ has small boundary for every $n$.

\begin{lem}\label{smallbdry}
If $u,v\in\partial D_n$, then $d_G(u,v)\leq 2r^2(R+2)+R = 4r^4-2r^3+4r^2+r-1$.
\end{lem}

\begin{proof}
Fix a correspondence $\cC\subset T\times G$ with stretch at most $r$. First we show that if $u,v\in\partial\tilde C_n$ are in the same component of $\tilde C_n^c$, then $d_G(u,v)\leq 2r^2(R+2)$. The result will then follow since $D_n^c$ consists of one component of $\tilde C_n^c$ together with $\{\sigma^{n}_1,\ldots,\sigma^{n}_{R}\}$.

Since $u$ and $v$ are in the same component of $\tilde C_n^c$, we may take a path $(u_0,u_1,\ldots,u_m)$ from $u$ to $v$ within $\tilde C_n^c$. Choose $x,y\in C_n$ such that $d_G(u,x)=R+1$ and $d_G(v,y)=R+1$; since by rule (a) $C_n$ is connected, we may also take a path $x_0,x_1,\ldots,x_p$ from $x$ to $y$ within $C_n$. Clearly $d_G(u_i,x_j)\geq R+1$ for each $i$ and $j$, and hence $d_T(\cC_{T,u_i},\cC_{T,x_j})\ge (R+2)/r-1$ for each $i$ and $j$.

Fix $a\in\cC_{T,u}$, $b\in\cC_{T,v}$, $t\in\cC_{T,x}$ and $t'\in\cC_{T,y}$. Note that for any $a_i\in\cC_{T,u_i}$ and $t_j\in\cC_{T,x_j}$, we have $d_T(a_i,a_{i+1})\le 2r-1$ and $d_T(t_j,t_{j+1})\le 2r-1$; therefore we can choose a path $\gamma_1$ from $a$ to $b$ and another $\gamma_2$ from $t$ to $t'$ such that
\[d_T(\gamma_1,\gamma_2)\ge \frac{R+2}{r}-1 - 2(r-1) = \frac{2r^2-r+1}{r}-2r+1 = 1/r > 0.\]
In particular, $\gamma_1$ and $\gamma_2$ do not intersect.

However, since $d(u,x)=R+1$ and $d(v,y)=R+1$, there exist paths $\gamma_3$ from $a$ to $t$ and $\gamma_4$ from $b$ to $t'$ of length at most $r(R+2)-1$ each. Note that $\gamma_2\cup\gamma_3\cup\gamma_4$ contains a path from $a$ to $b$, and since $T$ is a tree this path must be $\gamma_1$. But $\gamma_1$ and $\gamma_2$ do not intersect, so $\gamma_1$ must be contained in $\gamma_3\cup\gamma_4$, and therefore be of length at most $2r(R+2)-1$. Therefore $d(a,b)\le 2r(R+2)-1$, so $d(u,v)\le 2r^2(R+2)-1$ as required.
\end{proof}

\begin{lem}
There exist $\near,\adjust\in(0,1)$ (depending only on $\Delta$, $\cond$ and $r$) such that the sequence $D_1,\ldots,D_l$ is a valid strategy for~\Ptwo.
\end{lem}

\begin{proof}
To check rule (i), clearly $C_n\subset D_n$. Since $C_{n-1}\subset C_n$ by rule (a), to check that $D_{n-1}\subset D_n$ it suffices to check that $\{\sigma^n_1,\ldots,\sigma^n_R\}\subset D_{n-1}^c$. If $I_n\leq R-1$ this is clear; if $I_n\ge R$, then $d(C_n,s)\leq d(\sigma^{n-1}_R,s)\leq d(D_{n-1},s)$, so any shortest path from $C_n$ to $s$ does not intersect $D_{n-1}$.

For rule (ii), note that $p_{uv}\gtrsim_{\Delta,\cond} 1$ for all $uv\in E(G)$ and $\pi(v)\asymp_{\Delta,\cond} 1/|V(G)|$ for all $v\in V(G)$. It follows that since $\inbd D_n \subset \tilde C_n = B_G(R,C_n)$, there exists $\near$ depending only on $\Delta$, $\cond$ and $r$ such that $\inbd D_n$ is $\near$-near to $C_n$ for all $n$.

For rule (iii), we need to show that there exists $\adjust>0$ such that for any $S\subset D_{n}^c$ such that $C_n\cup S$ is connected,
\[Q((D_n\cup S)^c,C_n\cup S) \ge \adjust Q((D_n\cup S)^c,D_n\cup S).\]
If $S=D_n^c$ then this trivially holds since the right-hand side is $0$, so suppose $S\neq D_n^c$. Note that
\begin{align*}
Q((D_{n}\cup S)^c, D_{n}\cup S) &\leq Q((D_{n}\cup S)^c, D_{n}\setminus C_{n}) + Q((D_{n}\cup S)^c, C_{n}\cup S)\\
&\leq Q(D_{n}^c, D_{n}) + Q((D_{n}\cup S)^c,C_{n}\cup S),
\end{align*}
so it suffices to show that for small enough $\adjust$,
\[Q(D_{n}^c, D_{n}) \leq (1/\adjust-1)Q((D_{n}\cup S)^c, C_{n}\cup S).\]
But we know from Lemma \ref{smallbdry} that $Q(D_n^c, D_n)\lespar 1/|V(G)|$; and since there is at least one edge between $C_n$ and $D_n^c$ (namely that between $\sigma^n_0$ and $\sigma^n_1$), we have $Q((D_n\cup S)^c,C_n\cup S)\gtrsim_{\Delta,\cond} 1/|V(G)|$. Therefore we may choose $\adjust$ (depending on $\Delta$, $\cond$ and $r$) as required.

Rule (iv) is clearly satisfied.
\end{proof}

Our next lemma shows that our sequence $D_1,\ldots,D_l$ eats up the graph reasonably quickly.

\begin{lem}\label{eatquickly}
For any $n,k\ge1$ such that $n+k\le l$, we have
\[d_G(D_n,D_{n+k}^c) \geq k - (4r^4-2r^3+4r^2+r-1).\]
\end{lem}

\begin{proof}
By our construction of the sequence $D_1,\ldots, D_l$, we have $B_G(k,\{\sigma_0^n\}) \subset D_{n+k}$. Thus $d_G(\sigma_0^n,D_{n+k}^c)\geq k+1$. By Lemma \ref{smallbdry}, since $\sigma_0^n\sim \sigma_1^n\in \partial D_n$, for any $v\in \partial D_n$ we have
\[d_G(v,D_{n+k}^c)\geq d_G(\sigma_0^n,D_{n+k}^c) - 1 - d_G(\sigma_1^n,v) \geq k+1 - 1 - (4r^4-2r^3+4r^2+r-1).\]
Since $D_{n}^c$ is connected, the result follows.
\end{proof}

We now choose the $1$-bottleneck sequence on $T$ that we will compare to the sequence $D_1,\ldots,D_l$. Fix a correspondence $\cC\subset T\times G$ with stretch at most $r$.

Take any vertex $v_0\in C_1$, and then take a simple path $(t_0,\ldots,t_p)$ from $\cC_{T,v_0}$ to $\cC_{T,s}$. Then for $i=1,\ldots,p-1$, let $S_i = h_{t_p}(t_i)$, the set of vertices that are separated from $t_p$ by $t_i$.

It is easy to check that $S_1,\ldots,S_{p-1}$ is a $1$-bottleneck sequence for $T$. If we set $N_k = \min\{i : S_i^c\cap \bigcup_{v\in D_k}\cC_{T,v}=\emptyset\}$, then clearly $\bigcup_{v\in D_k} \cC_{T,v} \subset S_{N_k}$ and therefore $|D_k|\lesssim_{\Delta,r} |S_{N_k}|$. Let $K=4r^4-2r^3+8r^2-r-2$.

\begin{lem}\label{FctoSc}
For any $n\ge1$ and $k \ge K$ such that $n+k\leq l$, we have
\[\bigcup_{v\in D_{n+k}^c}\cC_{T,v} \subset S_{N_n}^c\]
and therefore
\[|D_{n+k}^c| \lesssim_{\Delta,r} |S_{N_n}^c|.\]
\end{lem}

\begin{proof}
Suppose there exists $v\in D_{n+k}^c$ such that $\cC_{T,v}\cap S_{N_n}\neq\emptyset$. Since $D_{n+k}^c$ is connected, we can choose a path $v=p_0,\ldots,p_m=s$ from $v$ to $s$ within $D_{n+k}^c$. For each $i$ choose $a_i\in\cC_{T,p_i}$, with $a_0\in S_{N_n}$ and $a_m=t_p$. Then $d_T(a_i,a_{i+1})\le 2r-1$, and since $T$ is a tree any path from $a_0$ to $a_m$ must pass through $t_{N_n}$. Therefore there exists $i$ such that $d_T(a_i,t_{N_n})\le r-1$.

On the other hand, by the definition of $N_n$, there exists $t\in S_{N_n}\setminus S_{N_n-1}$ and $u\in D_n$ such that $(t,u)\in\cC$. Since $D_n\cup\{\sigma^n_1,\ldots,\sigma^n_R\}$ is connected, we can choose a path $u=q_0,q_1,\ldots,q_{m'}=v_0$ from $u$ to $v_0$ such that $d_G(q_j,D_n)\le R$ for all $j$. For each $j$ take $b_j\in \cC_{T,q_j}$ with $b_0=t\in S_{N_n}\setminus S_{N_n-1}$ and $b_{m'}=t_0$. Then $d_T(b_j,b_{j+1})\le 2r-1$, and since $T$ is a tree any path from $b_0$ to $b_{m'}$ must pass through $t_{N_n}$. Therefore there exists $j$ such that $d_T(b_j,t_{N_n})\le r-1$.

Putting these two bounds together, we get $d_T(a_i,b_j)\le 2r-2$, so $d_G(p_i,q_j)\le r(2r-1)-1$. But $p_i\in D_{n+k}^c$ and $d_G(q_j,D_k)\le R$, so $d_G(D_k,D_{n+k}^c)\le r(2r-1)-1+R$. From Lemma \ref{eatquickly} we see that
\[k-(4r^4-2r^3+4r^2+r-1) \le d_G(D_k,D_{n+k}^c) \le r(2r-1)-1+R,\]
and rearranging (and recalling that $R=2r^2-r-1$) gives that $k\le K-1$, from which we deduce the result.
\end{proof}

Next we check that not too many of the sets $D_k$ can give rise to the same value of $N_k$. This follows easily from Lemma \ref{FctoSc}.

\begin{lem}\label{Ngrows}
For any $n\ge1$ and $k\geq K+1$ such that $n+k\le l$ we have $N_{n+k} > N_n$.
\end{lem}

\begin{proof}
We know from Lemma \ref{FctoSc} that if $k\ge K$, then
\[\bigcup_{v\in D_{n+k}^c} \cC_{T,v} \subset S^c_{N_n}.\]
But if $k\ge K+1$, then $D_{n+k}$ contains at least one element of $D_{n+K}^c$, and therefore $\bigcup_{v\in D_{n+k}} \cC_{T,v}\cap S^c_{N_n}\neq \emptyset$. Thus $N_{n+k}>N_n$.
\end{proof}

The next two results are simple technicalities that we will need in our proof of Proposition \ref{robustprop}.

\begin{lem}\label{offsums}
For any $j\in\mathbb{N}$,
\[\sum_{k=1}^l |D_k||D_k^c| \leq 2\sum_{k=1}^{l-j} |D_{k}||D_{k+j}^c| + \frac{j}{2}|V(G)|^2.\]
\end{lem}

\begin{proof}
Let $L=\max\{k : |D_k^c|\geq |V(G)|/2\}$. Then
\begin{align*}
\sum_{k=1}^{l-j} |D_{k}||D_{k+j}^c| &\geq \frac12\sum_{k=1}^{L-j} |D_k||D_k^c| + \frac12\sum_{k=L+1}^{l-j} |D_{k+j}||D_{k+j}^c|\\
&\geq \frac12\sum_{k=1}^l |D_k||D_k^c| - \frac12\sum_{k=L-j+1}^{L+j} |D_k||D_k^c|\\
&\geq \frac12\sum_{k=1}^l |D_k||D_k^c| - \frac12 (2j) \frac{|V(G)|^2}{4}.
\end{align*}
Rearranging gives the result.
\end{proof}

\begin{lem}\label{bnecklarge}
Provided $|V(T)|\geq 2$, there exists a connected set $A\subset V(T)$ such that $A$ and $A^c$ are both connected with $|A|\geq |V(T)|/(4\Delta)$ and $|A^c|\geq |V(T)|/2$.
\end{lem}

\begin{proof}
Take a connected set $B\subset V$ with $B^c$ also connected that maximises $|B||B^c|$. Suppose without loss of generality that $|B|\geq |V(T)|/2$. Choose a vertex $v\in B$ such that $v$ has a neighbour in $B^c$. Then at least one connected component of $B\setminus\{v\}$ must have size at least $(\frac{|V(T)|}{2}-1)\frac{1}{\Delta-1} \geq \frac{|V(T)|}{4\Delta}$. Call this component $A$. If $|A^c|<|V(T)|/2$, then $|B^c|<|A^c|<|V(T)|/2$ and $|B|>|A|>|V(T)|/2$, so $|A||A^c|>|B||B^c|$, contradicting the maximality of $|B||B^c|$. Therefore $|A^c|\geq |V(T)|/2$.
\end{proof}

Finally we are in a position to complete a proof of Proposition \ref{robustprop}.

\begin{proof}[Proof of Proposition \ref{robustprop}]
We have constructed a 1-bottleneck sequence $S_1,\ldots, S_{p-1}$ for $T$, and a valid strategy for~\Ptwo~such that any resulting sequence $D_1,\ldots,D_l$ satisfies $|D_n|\lesssim_{\Delta,r} |S_{N_n}|$ for each $n$ and, by Lemma \ref{FctoSc}, $|D_{n+K}^c| \lesssim_{\Delta,r} |S_{N_n}^c|$ (provided $n+K\leq l$). Applying Lemma \ref{offsums}, we obtain
\[\sum_{n=1}^l |D_n||D_n^c| \leq 2\sum_{n=1}^{l-K} |D_n||D_{n+K}^c| + \frac K 2 |V(G)|^2 \lesssim_{\Delta,r} \sum_{n=1}^{l-K} |S_{N_n}||S_{N_n}^c| + |V(T)|^2.\]
By Lemma \ref{Ngrows}, we have
\[\sum_{n=1}^{l-K} |S_{N_n}||S_{N_n}^c| \leq (K+1)\sum_{n=1}^{p-1} |S_n||S_n^c| \leq (K+1)\max_{(S_1,\ldots,S_{m})\in\Sc_1(T)}\sum_{j=1}^{m} |S_j||S_j^c|,\]
and by Lemma \ref{bnecklarge} we have
\[\max_{(S_1,\ldots,S_{m})\in\Sc_1(T)}\sum_{j=1}^{m} |S_j||S_j^c| \geq \frac{|V(T)|^2}{8\Delta}.\]
Putting these together, we see that
\[\sum_{n=1}^l |D_n||D_n^c| \lesssim_{\Delta,r} \max_{(S_1,\ldots,S_{m})\in\Sc_1(T)}\sum_{j=1}^{m} |S_j||S_j^c|\]
as required.
\end{proof}

\subsection*{Acknowledgements}
LAB was partially supported by NSERC Discovery Grant 341845. MR was partially supported by an EPSRC fellowship EP/K007440/1. Both authors would like to thank several anonymous referees for their careful reading of the manuscript.

\bibliographystyle{plain}

\begin{thebibliography}{10}

\bibitem{aldous:some}
David Aldous.
\newblock Some inequalities for reversible {M}arkov chains.
\newblock {\em J. London Math. Soc}, 25(2):564--576, 1982.

\bibitem{aldous:reversible}
David Aldous and James~Allen Fill.
\newblock Reversible {M}arkov chains and random walks on graphs, 2002.
\newblock Unfinished monograph, recompiled 2014, available at
  \texttt{http://www.stat.berkeley.edu/$\sim$aldous/RWG/book.html}.

\bibitem{benjamini:coarse_geom}
Itai Benjamini.
\newblock Coarse geometry and randomness.
\newblock In {\em \'Ecole d'\'Et\'e de Probabilit\'es de Saint-Flour
  XLI---2011}, volume 2100 of {\em Lecture Notes in Math.} Springer, Berlin,
  2013.

\bibitem{diestel:connected_tree_width}
Reinhard Diestel and Malte M{\"u}ller.
\newblock Connected tree-width.
\newblock To appear in \emph{Combinatorica}. Preprint:
  \texttt{http://arxiv.org/abs/1211.7353}.

\bibitem{ding_lee_peres:cover_times}
Jian Ding, James~R. Lee, and Yuval Peres.
\newblock Cover times, blanket times, and majorizing measures.
\newblock {\em Annals of Mathematics}, 175(3):1409--1471, 2012.

\bibitem{ding_peres:sensitivity_mixing}
Jian Ding and Yuval Peres.
\newblock Sensitivity of mixing times.
\newblock {\em Electronic Communications in Probability}, 18:1--6, 2013.

\bibitem{fountoulakis_reed:faster_mixing}
Nikolaos Fountoulakis and Bruce~A. Reed.
\newblock Faster mixing and small bottlenecks.
\newblock {\em Probability Theory and Related Fields}, 137(3):475--486, 2007.

\bibitem{goel_et_al:mixing_spectral}
Sharad Goel, Ravi Montenegro, and Prasad Tetali.
\newblock Mixing time bounds via the spectral profile.
\newblock {\em Electron. J. Probab}, 11(1):1--26, 2006.

\bibitem{griffiths_et_al:tight_hitting_times}
Simon Griffiths, Ross Kang, Roberto Oliveira, and Viresh Patel.
\newblock Tight inequalities among set hitting times in {M}arkov chains.
\newblock {\em Proceedings of the American Mathematical Society},
  142(9):3285--3298, 2014.

\bibitem{hermon:robustness_mixing}
Jonathan Hermon.
\newblock On sensitivity of uniform mixing times.
\newblock To appear in Ann. Inst. Henri Poincar\'e Probab. Stat. Preprint:
  \texttt{http://arxiv.org/abs/1607.01672}.

\bibitem{hermon_peres:characterization}
Jonathan Hermon and Yuval Peres.
\newblock A characterization of ${L}_2$ mixing and hypercontractivity via
  hitting times and maximal inequalities.
\newblock To appear in Probab.~Theory Relat.~Fields. Preprint:
  \texttt{http://arxiv.org/abs/1609.07557}.

\bibitem{jerrum1988conductance}
Mark Jerrum and Alistair Sinclair.
\newblock Conductance and the rapid mixing property for markov chains: the
  approximation of permanent resolved.
\newblock In {\em Proceedings of the twentieth annual ACM symposium on Theory
  of computing}, pages 235--244. ACM, 1988.

\bibitem{kozma:precision_spectral_profile}
Gady Kozma.
\newblock On the precision of the spectral profile.
\newblock {\em ALEA}, 3:321--329, 2007.

\bibitem{levin_et_al:MCs_mixing}
D.A. Levin, Y.~Peres, and E.L. Wilmer.
\newblock {\em Markov chains and mixing times}.
\newblock American Mathematical Society, 2009.

\bibitem{lovasz_kannan:faster_mixing}
L{\'a}szl{\'o} Lov{\'a}sz and Ravi Kannan.
\newblock Faster mixing via average conductance.
\newblock In {\em Proceedings of the thirty-first annual ACM symposium on
  Theory of computing}, pages 282--287. ACM, 1999.

\bibitem{lovasz_winkler:efficient_stopping_rules}
L{\'a}szl{\'o} Lov{\'a}sz and Peter Winkler.
\newblock Efficient stopping rules for {M}arkov chains.
\newblock In {\em Proceedings of the twenty-seventh annual ACM symposium on
  Theory of computing}, pages 76--82. ACM, 1995.

\bibitem{moon:random_walks_trees}
J.W. Moon.
\newblock Random walks on random trees.
\newblock {\em Journal of the Australian Mathematical Society}, 15(01):42--53,
  1973.

\bibitem{morris_peres:evolving_sets}
Ben Morris and Yuval Peres.
\newblock Evolving sets, mixing and heat kernel bounds.
\newblock {\em Probability Theory and Related Fields}, 133(2):245--266, 2005.

\bibitem{oliveira:mixing}
Roberto~Imbuzeiro Oliveira.
\newblock Mixing and hitting times for finite {M}arkov chains.
\newblock {\em Electron. J. Probab}, 17(70):1--12, 2012.

\bibitem{peres:mixing}
Yuval Peres and Perla Sousi.
\newblock Mixing times are hitting times of large sets.
\newblock {\em Journal of Theoretical Probability}, pages 1--32, 2013.

\end{thebibliography}
\def\cprime{$'$}

\end{document}